\documentclass{article}
\usepackage{amsmath}
 \usepackage{amssymb}
\usepackage{amsthm}
 \usepackage{graphicx}
 \usepackage{psfrag}
 \usepackage{epsfig}
\newtheorem{thm}{Theorem}[section]  
\newtheorem{lem}{Lemma}[section]
\newtheorem{prop}{Proposition}[section]
\newtheorem{rem}{Remark}[section]

\def\X{\mathcal{X}}
\def\T{\mathcal{T}}
\def\R{\mathbb{R}}
\def\N{\mathbb{N}}

\begin{document}


%
%

\title{CONVERGENCE OF A GREEDY ALGORITHM FOR HIGH-DIMENSIONAL CONVEX NONLINEAR PROBLEMS}

\author{Eric Canc\`es, Virginie Ehrlacher and Tony Leli\`evre}

\maketitle
\begin{center}
Universit\'e Paris Est, \'Ecole des Ponts - Paristech, CERMICS,\\ 
6 \& 8 avenue Blaise Pascal\\ 77455 Marne-la-Vall\'ee Cedex 2\\ FRANCE\\[3pt]
\{cances,ehrlachv,lelievre\}@cermics.enpc.fr\\
\end{center}


\begin{abstract}
In this article, we present a greedy algorithm based on a tensor product decomposition, whose aim is to compute the global minimum of a strongly convex energy functional. We 
prove the convergence of our method provided that the gradient of the energy is Lipschitz on bounded sets. The main interest of this  
method is that it can be used for high-dimensional nonlinear convex problems. We illustrate this method on a prototypical example for uncertainty 
propagation on the obstacle problem.
\end{abstract}


\section{Introduction}

The main motivation for this work comes from two important and challenging problems in contemporary scientific computing:
\begin{itemize}
\item the uncertainty quantification for some nonlinear models in mechanics, and more precisely, for contact problems;
\item the computation of some high-dimensional functions in molecular dynamics.
\end{itemize}

Concerning the first domain of application which is the main focus of this work, there is now a wide literature
 on the subject, ranging from specific questions related to the modeling of the noise sources (in particular of their correlation),
 to dedicated methods for the evaluation of events with very small probabilities (reliability). The focus of this paper is rather
 on the development of methods to compute efficiently a {\em reduced model} which rapidly gives the output of interest as a function of the random
 variable which enters the input parameters,  in the context of contact problems in continuum mechanics. Such a model can then be used to evaluate
 the distribution of the outputs (for a given distribution of the input parameters), or to reduce the variance in a Monte Carlo computation for example. 
Many methods have been proposed in the literature to attack this problem\cite{NouyReview,Ghanem}: stochastic collocation methods, Galerkin
 methods, perturbation methods, etc. To be more specific, let us assume that the noise on the parameters of the model can be modeled by a possibly large number
 of random variables $T=(T_1, \ldots, T_p) \in \mathbb{R}^p$, so that the quantity of interest (say the displacement field) $u(t,x)$ is a function of $(p+d)$ variables,
 where $d$ is the dimension of the physical space. The question is then how to approximate a function on such a high-dimensional space. The natural idea at
 the basis of many methods is to look for the solution to this problem as a linear combination of tensor products:
$$
u(t,x)=\sum_{i=1}^l \sum_{j=1}^m U^{ij} \, \phi_i(t) \psi_j(x),
$$
where $(\phi_i)_{1 \leq i \leq l}$ and $(\psi_j)_{1 \leq j \leq m}$ are bases of vector spaces of dimension $l$ and $m$ respectively which are fixed 
a priori, and where $(U^{ij})_{1 \leq i \leq l, 1\leq j \leq m}$ are real numbers to be computed. 
This method leads to the resolution of a problem in a vector space of dimension $N=lm$ which may be very large.
 This difficulty becomes all the more pregnant if $p$ is really large, so that the solution should be  typically approximated as a sum:
\begin{equation}
\label{devlpt}
 u(t,x)=\sum_{i_1=1}^l ... \sum_{i_p=1}^l \sum_{j=1}^mU^{i_1,..., i_p,j} \, \phi^1_{i_1}(t_1) \ldots \phi^p_{i_p}(t_p) \psi_j(x).
\end{equation}
In this case, $N=l^p m$ will be too large for a classical discretization method. The method we are studying is a way to circumvent this difficulty.

The second application we have in mind is the computation of the solution to a high-dimensional Poisson equation arising in molecular dynamics,
 called the committor function\cite{Weinan}. Mathematically, this function gives the probability for a stochastic process to reach a given region 
(say $A \subset \mathbb{R}^d$) before another one (say $B \subset \mathbb{R}^d$).
 Using Feynman-Kac formula, it can be shown that this function satisfies a Poisson equation in a weighted Sobolev space, with Dirichlet boundary conditions
 (namely $1$ on $A$ and $0$ on $B$).  Typically, the stochastic process lives in a high-dimensional space ($d$ is large), so that computing this function is a challenge.

In both cases, the difficulty comes from the high-dimensionality of the function to approximate. 
The principle of the method we are interested in is:
 (i) to rewrite the original problem as a minimization problem:
\begin{equation}\label{eq:minimisation}
 u \in \mathop{\mbox{argmin}}_{v \in V} {\mathcal E}(v)
\end{equation}
where ${\mathcal E}$ is a functional defined on some Hilbert space $V$ and (ii) to expand the solution in tensor products of lower-dimensional functions 
\begin{equation}\label{eq:approx}
 u_n(t,x)=\sum_{k=1}^n r_k(t) s_k(x).
\end{equation}
In practice, for each $k$, the functions $r_k$ and $s_k$ are computed as linear combinations of the functions of the bases $(\phi_i)_{1\leq i \leq l}$ and 
$(\psi_j)_{1\leq j\leq m}$ so that
\begin{equation}\label{eq:formrk}
r_k(t)  = \sum_{i=1}^l R_k^i \phi_i(t),
\end{equation}
and
\begin{equation}\label{eq:formsk}
s_k(x) = \sum_{j=1}^m S_k^j \psi_j(x),
\end{equation}
where for each $k\in\N^*$, $R_k = (R_k^i)_{1\leq i\leq l} \in \R^l$ and $S_k = (S_k^j)_{1\leq j\leq m} \in \R^m$. 
In the end, computing the approximation (\ref{eq:approx}) leads to a problem of dimension $\widetilde{N} = n(l+m)$ which, provided that $n$ remains small enough, will hopefully 
be lower than the dimension of the problem obtained with the classical approach $N=lm$ when the size of the bases $l$ and $m$ are large.

The reduction of dimension is even more significant when we are in the case of equation (\ref{devlpt}). Indeed, the approximation (\ref{eq:approx}) can be adapted in this case in the following form:
$$
 u_n(t,x)=\sum_{k=1}^n r^1_k(t_1)\cdots r^p_k(t_p) s_k(x).
$$
In this case, the overall dimension of the problem will be $\widetilde{N} = n(pl + m)$ instead of $N = l^pm$ in the classical approach.
\medskip

Such a representation of a function as a sum of tensor product of other functions to avoid the curse of dimensionality have already been introduced in the literature. 
One approach consists in using the so-called sparse tensor product representation\cite{Smolyak,Schwab,Griebel}. If the solution $u$ we wish to approximate is sufficiently regular, one does not 
need to use fine discretizations in each direction. This idea can be used for example in Galerkin-like discretizations. However, this method loses its efficiency in the case when the solution $u$ is not 
regular enough or when the mesh considered is complicated.

\medskip 

We adopt another approach in this article.
The principle of our method is to determine {\em sequentially} the pairs of functions $(r_k,s_k)$ which intervene in the approximation (\ref{eq:approx}) through 
the following minimization problem:
\begin{equation}\label{eq:min_pb}
(r_n,s_n) \in \mathop{\mbox{argmin}}_{(r,s) \in V_t \times V_x} {\mathcal E}\left(\sum_{k=1}^{n-1}  r_k(t) s_k(x) + r(t)s(x) \right),
\end{equation}
where $V_t$ and $V_x$ in (\ref{eq:min_pb}) denote respectively Hilbert spaces of functions depending only on the variable $t$ or only on the variable $x$.

To rewrite the two problems mentioned above as minimization problems on Hilbert spaces, we penalize the constraints, namely the presence of the obstacle
 for the contact problem, or the Dirichlet boundary conditions for the high-dimensional Poisson problem.

The method described above has been introduced by Chinesta\cite{Chinesta} for solving high-dimensional Fokker-Planck equations, by Nouy\cite{Nouy09} in the context of uncertainty
 quantification in mechanics, and is very much related to so-called greedy algorithms\cite{Temlyakov,LeBr09} used in nonlinear approximation theory.
 The main contributions 
of this work are the following:
\begin{itemize}
\item the convergence of the greedy algorithm (\ref{eq:approx})-(\ref{eq:min_pb}) to the unique solution of (\ref{eq:minimisation}) is proved, under the key assumptions that ${\mathcal E}$ is strongly convex and that the gradient of $\mathcal{E}$ is 
Lispchitz on bounded sets;
\item an exponential rate of convergence is obtained in the finite dimensional case;
\item an adequate procedure to solve the minimization subproblem~\eqref{eq:min_pb} is proposed and tested on an academic test case.
\end{itemize}
This paper can be seen as an extension of previous works on greedy algorithms \cite{Temlyakov,LeBr09} which concentrate on the linear case,
 namely when ${\mathcal E}(v)=\frac{1}{2}\|v\|_V^2 - L(v)$, where $\|\cdot\|_V$ is the norm of the Hilbert space $V$, and $L$ a continuous linear form on $V$.

We would like to stress that even if all the results and proofs are provided in the context of tensor products of two functions, our results
 can be easily generalized to the case of tensor products of more than two functions such as (\ref{devlpt}) {\em except for the results in
 Section~5}. We have chosen not to present the results in this general setting for the sake of clarity.

The paper is organized as follows. In Section~2, we introduce the general setting for the problem we consider and state the main result
 of this paper, namely the convergence of the greedy algorithm. Section~2 also presents more precisely the two specific examples of application
 we have in mind. Section~3 is devoted to the proof of the convergence. In Section~4, an exponential rate of convergence is proved, in the finite dimensional setting (i.e. when $V_t$ and $V_x$
 are finite dimensional spaces). Section~5 shows that, under specific additional assumptions
 which are typically satisfied in the context of uncertainty quantification, the convergence results also hold if $(r_n,s_n)$ in~\eqref{eq:min_pb}
 is only a local minimum. Finally, Section~6 is devoted to a discussion of the numerical implementation, as well as to the presentation of test cases on a toy model.

\section{Presentation of the problem and the convergence result}

In this paper, we are interested in the convergence of a greedy algorithm for the minimization of high-dimensional nonlinear convex problems. 

We first introduce the general theoretical setting in which we prove the convergence, then describe two prototypical examples to which our analysis can be applied.

\subsection{General theoretical setting}

Throughout this article, $p$ and $d$ denote some positive integers, and $\mathcal{T}$ and $\mathcal{X}$ some open sets of $\mathbb{R}^p$ and
 $\mathbb{R}^d$ respectively.

Let $V_t$ and $V_x$ be Hilbert spaces of real-valued functions respectively
 defined over $\mathcal{T}$ and $\mathcal{X}$ (typically $L^2$ or Sobolev spaces). Let $\| . \|_t$ and $\| . \|_x$ be the norms of $V_t$ and $V_x$.

We introduce the following tensor product for all $(r,s) \in V_t\times V_x$, 
\begin{equation}
\label{tensor}
r\otimes s :\left\{
\begin{array}{ccc}
 \mathcal{T}\times \mathcal{X} & \rightarrow & \mathbb{R} \\
(t,x) & \mapsto & r(t)s(x)\\
\end{array}\right. ,
\end{equation}
which defines a real-valued function on $\mathcal{T} \times \mathcal{X}$.

We also denote by $\Sigma = \left\{ r\otimes s\; | \; (r, s)\in V_t \times V_x\right\}$.

Let $V$ be a Hilbert space of real-valued functions defined on $\mathcal{T}\times \mathcal{X}$. The 
scalar product of $V$ is denoted by $\langle .,.\rangle$ and the associated norm by $\|.\|_V$.

Let $\mathcal{E}$ be a differentiable real-valued functional defined on $V$. 
 For all $v\in V$, we denote by $\mathcal{E}'(v)$ the 
gradient of $\mathcal{E}$ at $v$.

We make the following assumptions:

\begin{description}
 \item { $(A1)$ } $\mbox{Span}(\Sigma)$ is a dense subset of $V$ for $\|.\|_V$;
\item { $(A2)$ } for all sequences of $\Sigma$ bounded in $V$, there exists a subsequence which weakly converges in $V$ towards an element of $\Sigma$;
\item{ $(A3)$ } the functional $\mathcal{E}$ is strongly convex for $\|.\|_V$, i.e. there exists 
a constant $\alpha\in \mathbb{R}_+^*$ for which
\begin{equation}
 \label{alpha}
\forall v,w\in V,\; \mathcal{E}(v) \geq \mathcal{E}(w) + \langle \mathcal{E}'(w),v-w\rangle + \frac{\alpha}{2} \|v-w\|_V^2.
\end{equation}
The functional $\mathcal{E}$ is also said to be $\alpha$-convex;

\item{ $(A4)$ } the gradient of $\mathcal{E}$ is Lipschitz on bounded sets: 
for each bounded subset $K$ of $V$, there exists a nonnegative constant $L_K\in\mathbb{R}_+$ such that
\begin{equation}
\label{Lipschitz}
 \forall v, w\in V,\quad \|\mathcal{E}'(v) - \mathcal{E}'(w)\|_V \leq L_K\|v-w\|_V.
\end{equation}

\end{description}

The unique global minimizer of $\mathcal{E}$ on $V$ is denoted by $u$. Its existence and uniqueness are ensured by the $\alpha$-convexity of the functional $\mathcal{E}$:
$$u  = \mathop{\mbox{argmin}}_{v\in V} \mathcal{E}(v).$$

\vspace{0.5cm}
We are going to study the following algorithm: the sequence $\left((r_n, s_n)\right)_{n\in\mathbb{N}^*} \in \left(V_t \times V_x\right)^{\mathbb{N}^*}$ is defined recursively by
\begin{equation}
\label{algorithm}
(r_n,s_n) \in \mathop{\mbox{argmin}}_{(r,s)\in V_t\times V_x} \mathcal{E}\left(\sum_{k=1}^{n-1} r_k \otimes s_k +r \otimes s\right).
\end{equation}

Throughout this article, we will denote for all $n\in\mathbb{N}^*$,
\begin{equation}
 \label{undef}
u_n = \sum_{k=1}^n r_k \otimes s_k .
\end{equation}

Our main result is the following theorem, whose proof is given in Section~3.

\begin{thm}
Under the assumptions $(A1)$, $(A2)$, $(A3)$ and $(A4)$, the iterations of the algorithm are well-defined, in the sense that (\ref{algorithm}) has at least one 
minimizer $(r_n, s_n)$. Moreover, 
the sequence $(u_n)_{n\in\mathbb{N}}$ strongly converges in $V$ towards $u$.  
\end{thm}

\vspace{0.5cm}
\begin{rem}{\normalfont For each $n\in\mathbb{N}^*$, the minimizer of (\ref{algorithm}) is not unique in general. In particular, notice that the function 
$V_t\times V_x \ni (r,s) \mapsto \mathcal{E}\left(\sum_{k=1}^{n-1} r_k \otimes s_k +r \otimes s\right)$ is not convex.}
\end{rem}

\vspace{0.5cm}
\begin{rem}{\normalfont Theorem~2.1 could be generalized to the case of tensor products of more than two Hilbert spaces. 

Indeed, let $q\in\mathbb{N}$ with 
$q\geq 3$. Let $p_1,\cdots, p_q$ be $q$ positive integers. Let $\mathcal{T}_1,\cdots,\mathcal{T}_q$ be $q$ open subsets
 of $\mathbb{R}^{p_1},\cdots,\mathbb{R}^{p_q}$ respectively. We consider  $q$ Hilbert spaces, $V_1,\cdots, V_q$ of real-valued functions defined
 respectively on $\mathcal{T}_1,\cdots,\mathcal{T}_q$. Let $V$ be a Hilbert space of real-valued functions defined on $\mathcal{T}_1 \times \cdots \times \mathcal{T}_q$. 
Let $\mathcal{E}$ be a real-valued differentiable functional defined on $V$. We denote by 
$\Sigma = \left\{ r^{(1)} \otimes \cdots \otimes r^{(q)}\; | \; \left(r^{(1)}, \cdots, r^{(q)}\right) \in V_1 \times \cdots \times V_q \right\}$.
 Our algorithm can then easily be adapted provided that assumptions $(A1)$, $(A2)$, $(A3)$ and $(A4)$ are satisfied:
 $\left(r_n^{(1)}, \cdots, r_n^{(q)}\right) \in V_1 \times \cdots \times V_q$ are defined recursively by
$$\left(r_n^{(1)}, \cdots, r_n^{(q)}\right) \in \mathop{\mbox{argmin}}_{\left(r^{(1)},\cdots, r^{(q)}\right) \in V_1 \times \cdots \times V_q } 
\mathcal{E} \left( \sum_{k=1}^{n-1} r_k^{(1)} \otimes \cdots \otimes r_k^{(q)} +  r^{(1)} \otimes \cdots \otimes r^{(q)} \right).$$ 

Our convergence result also holds in this case. But for the sake of simplicity, we will limit our analysis to the case of only two Hilbert spaces.}
\end{rem}

\begin{rem}{\normalfont
Let $(.,.)$ be the scalar product defined on $\mbox{Span}(\Sigma)$ as: for all $(r_1,r_2,s_1,s_2)\in V_t^2 \times V_x^2$, 
$$(r_1\otimes s_1, r_2 \otimes s_2 ) = \langle r_1, r_2 \rangle_{V_t} \langle s_1, s_2 \rangle_{V_x},$$
where $\langle ., . \rangle_{V_t}$ and $\langle .,.\rangle_{V_x}$ denote the scalar products of $V_t$ and $V_x$ respectively.
Let $\|.\|$ be the cross-norm associated to the scalar product $(.,.)$. The tensor space of $V_t$ and $V_x$, denoted as $V_t\otimes V_x$ is then defined as 
the closure of $\mbox{Span}(\Sigma)$ for the product norm $\|.\|$,
$$V_t\otimes V_x = \overline{\mbox{Span}(\Sigma)}^{\|.\|}.$$

Let us point out that the Hilbert space $V$ is not necessarily equal to $V_t\otimes V_x$, the tensor space of $V_t$ and $V_x$ associated to the tensor product (\ref{tensor}). 
Indeed, an example where our analysis can be applied and where $V\neq V_t \otimes V_x$ is given in Section~2.2.2 (see Remark~2.5.). However, the following inclusion relationship holds: $V_t\otimes V_x \subset V$.}
\end{rem}

\begin{rem}{\normalfont If $V_t$ and $V_x$ are discretized in finite-dimensional spaces of dimension $l$ and $m$, our algorithm consists in solving several problems in dimension $l+m$ 
instead of solving one problem of dimension $lm$. Thus, we can circumvent the curse of high-dimensionality.} 
\end{rem}

\subsection{Prototypical problems} 

To prove that the general theoretical setting we described in Section~2.1 is satisfied on the prototypical problems we present in this section, we need the following lemma,
 which is well-known in distribution theory\cite{Schwa66}. 

\begin{lem}
Let $U\in\mathcal{D}'(\mathcal{T}\times\mathcal{X})$ be a distribution such that for any functions $(\phi, \psi)\in\mathcal{C}^{\infty}_c(\mathcal{T}) \times \mathcal{C}^{\infty}_c(\mathcal{X})$, 
$$\left( U, \phi\otimes\psi\right)_{\left(\mathcal{D}'(\mathcal{T}\times\mathcal{X}), \mathcal{D}(\mathcal{T}\times\mathcal{X})\right)} = 0 .$$

Then $U=0$ in $\mathcal{D}'(\mathcal{T}\times\mathcal{X})$. Moreover, for any two sequences of distributions $R_n\in\mathcal{D}'(\mathcal{T})$ and 
$S_n\in \mathcal{D}'(\mathcal{X})$ such that $\lim_{n\rightarrow\infty} R_n = R$ in $\mathcal{D}'(\mathcal{T})$ and $\lim_{n\rightarrow\infty}S_n = S$ 
in $\mathcal{D}'(\mathcal{X})$, $\lim_{n\rightarrow \infty} R_n \otimes S_n = R \otimes S$ in $\mathcal{D}'(\mathcal{T}\times\mathcal{X})$.

\end{lem}

\subsubsection{Uncertainty propagation on obstacle problems}

\vspace{0.5cm}
An example of application of our algorithm is the study of uncertainty propagation on obstacle problems. We assume
 that uncertainty can be modeled by a set of $p$ random variables $T_1$, $T_2$, ..., $T_p$, and that the random vector $T=(T_1, ..., T_p)$ takes its values in $\mathcal{T}$.

We consider also that the physical problem is defined over the domain $\mathcal{X}$, which is supposed to be a bounded subset of $\mathbb{R}^d$. If $H$ is a Hilbert
 space of functions defined on $\mathcal{X}$, we denote by
$$L^2_T(\mathcal{T}, H) = \left\{ v:\mathcal{T} \rightarrow H \; | \; \mathbb{E}\left[ \| v(T) \|_H^2 \right] < + \infty \right\},$$
where $\mathbb{E}$ denotes the expectation with respect to the probability law of $T$, and $\|.\|_H$ denotes the norm of $H$. 
We endow $L^2_T(\mathcal{T}, H)$ with the scalar product defined by $\langle v,w\rangle_{L^2_T(\mathcal{T}, H)} = \mathbb{E} \left[ \langle v(T), w(T) \rangle_H \right]$.

\vspace{0.5cm}

A formulation of the obstacle problem with uncertainty is the following\cite{Gros07}. Let $g\in L^2_T(\mathcal{T}, H^1_0(\mathcal{X}))$ and $f\in L^2_T(\mathcal{T}, H^{-1}(\mathcal{X}))$. 
A membrane is stretched over the domain $\mathcal{X}$ and is deflected by some random force having pointwise density $f(T,x)$ for $x\in\mathcal{X}$. At the boundary $\partial \mathcal{X}$, 
the membrane is fixed and in the interior of $\mathcal{X}$ the deflection is assumed to be bounded from below by the function $g(T,x)$ (a random obstacle). Then the deflection 
$z = z(T,x)$ is solution of the following obstacle problem with uncertainty (see figure~1):
\begin{equation}
\label{form1}
 \left\{
\begin{array}{cc}
\left .
\begin{array}{c}
 -\Delta_x z(t,x) \geq f(t,x) \\
z(t,x) \geq g(x,t)\\
\left( \Delta_x z(t,x) + f(t,x) \right) \left( z(t,x) - g(t,x) \right)  =  0\\
\end{array} \right\} & \; \mbox{for a.a.}\; (t,x) \in \mathcal{T}\times \mathcal{X}, \\
z(t,x) = 0 & \; \mbox{for a.a.}\;(t,x) \in \mathcal{T} \times \partial \mathcal{X}.
\end{array}\right .
\end{equation}

\begin{figure}[h]
\centerline{\psfig{file=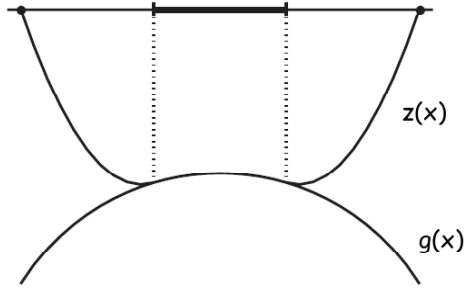,width=8.5cm}}
\vspace*{8pt}
\caption{Obstacle problem.}
\end{figure}

An equivalent formulation of this problem is the following. Let us denote
$$\mathcal{K}_g = \left\{ v\in L^2_T(\mathcal{T}, H^1_0(\mathcal{X}))\; | \; \mbox{for a.a.}\; (t,x) \in \mathcal{T} \times \mathcal{X}, \; v(t,x) \geq g(t,x) \right\}.$$

Solving the obstacle problem~(\ref{form1}) consists in solving the minimization problem
\begin{equation}
\label{obstacle}
\mathop{\inf}_{v\in\mathcal{K}_g} \mathcal{J}(v),
\end{equation}
where $\mathcal{J}(v) = \mathbb{E}\left[ \frac{1}{2} \int_{\mathcal{X}} |\nabla_x v(T,x)|^2dx - \left\langle f(T,.), v(T,.)\right\rangle_{H^{-1}(\mathcal{X}), H^1_0(\mathcal{X})}\right]$.

\vspace{0.5cm}

One of the main difficulties of this kind of problems is their very high nonlinearity. Many methods have been proposed to approximate the solution of these problems in the case without 
uncertainty\cite{Glow76,For82,Glow89,Gros07}.
Among them, penalization methods\cite{Gros07,Glow76} are among the most widely used. They consist in approximating the solution of a given obstacle
 problem by a sequence of solutions of penalized 
problems defined on the entire Hilbert space. 

Let $\rho$ be a parameter in $\mathbb{R}_+$.
Such a penalized problem associated with problem~(\ref{obstacle}) may be defined as
\begin{equation}
\label{penalized}
\mathop{\inf}_{v\in L^2_T(\mathcal{T}, H^1_0(\mathcal{X}))} \mathcal{J}_{\rho}(v),
\end{equation}
where $\mathcal{J}_{\rho}(v) = \mathcal{J}(v) + \mathbb{E} \left[ \frac{\rho}{2} \int_{\mathcal{X}} [g(T,x)-v(T,x)]_+^2 dx\right]$.

Here and below, we denote by $\left[a\right]_+$ the positive part of the real number $a$, i.e. $[a]_+ = 0$ if $a\leq 0$ and $[a]_+ = a$ if $a\geq 0$.

When $\rho$ goes to infinity, the solution $z_{\rho}$ of problem~(\ref{penalized}) strongly converges to the solution $z$ of problem~(\ref{obstacle}).
The goal of the algorithm we described in the previous section is to calculate the solution $u=z_{\rho}$ of this regularized problem for a given value of the parameter $\rho$.

\vspace{0.5cm}

Let us check that the general theoretical setting we described in Section~2.1 can be applied in this case. 

Let us consider $V =  L^2_T(\mathcal{T}, H_0^1(\mathcal{X}))$, $V_t = L^2_T(\mathcal{T}, \mathbb{R})$, $V_x=  H_0^1(\mathcal{X})$ and $\mathcal{E}(v) = \mathcal{J}_{\rho}(v)$ for $v\in V$. 
We have $\Sigma = \left\{ r\otimes s \; | \; (r,s)\in V_t \times V_x\right\}$. We endow $H_0^1(\mathcal{X})$ with the scalar product defined by
 $\langle s_1, s_2 \rangle_{H_0^1(\mathcal{X})} = \int_{\mathcal{X}} \nabla s_1(x).\nabla s_2(x) dx$.
\vspace{0.2cm}
In this case, we have $V = V_t \otimes V_x$ and as a consequence assumption $(A1)$ is obviously satisfied. 

\vspace{0.2cm}
Besides, assumption $(A2)$ is satisfied as well. If $((r_n, s_n))_{n\in\mathbb{N}} \in (V_t\times V_x)^{\mathbb{N}}$ is such that $(\|r_n\otimes s_n\|_V)_{n\in\mathbb{N}}$ is bounded, 
it is possible to extract a subsequence which weakly converges in $V$ towards an element $w\in V$. Besides, 
there exists a non-negative constant $C\in\mathbb{R}_+$ such that for all $n\in\mathbb{N}$,
\begin{eqnarray*}
\|r_n\otimes s_n\|_V^2 & = & \mathbb{E}\left[ \int_{\mathcal{X}} |\nabla_x \left(r_n\otimes s_n\right) (T,x)|^2dx \right] \\
& = & \mathbb{E}\left[|r_n(T)|^2\right] \int_{\mathcal{X}} |\nabla_x s_n(x)|^2 dx \\
& = & \|r_n\|_{V_t}^2 \|s_n\|_{V_x}^2 \\
& \leq & C .\\
\end{eqnarray*}

We can then choose $((r_n^*, s_n^*))_{n\in\mathbb{N}} \in (V_t\times V_x)^{\mathbb{N}}$ such that $r_n^* \otimes s_n^* = r_n \otimes s_n$ and $\|r_n^*\|_{L^2_T(\mathcal{T}, \mathbb{R})} = 1$.
The sequences $(r_n^*)_{n\in\mathbb{N}}$ and $(s_n^*)_{n\in\mathbb{N}}$ are then bounded in $L^2_T(\mathcal{T}, \mathbb{R})$ and $H_0^1(\mathcal{X})$ respectively
and we can extract subsequences which weakly converge in $L^2_T(\mathcal{T}, \mathbb{R})$ and $H_0^1(\mathcal{X})$ towards $r^{\infty}\in L^2_T(\mathcal{T}, \mathbb{R})$ and $s^{\infty}\in H_0^1(\mathcal{X})$ respectively. As the weak convergences in $L^2_T(\mathcal{T}, \mathbb{R})$ and $H_0^1(\mathcal{X})$ imply 
the convergences in the distributional sense, the sequence $r_n^* \otimes s_n^* = r_n \otimes s_n$ necessarily converges towards $r^{\infty}\otimes s^{\infty}$
 in $\mathcal{D}'(\mathcal{T}\times \mathcal{X})$ by Lemma 2.1. 
As the weak convergence in $V$ also implies the convergence in the sense of the distributions, we obtain, by uniqueness of the limit,
 $w=r^{\infty} \otimes s^{\infty}\in \Sigma$. Hence assumption $(A2)$ is satisfied.

The functional $\mathcal{E}$ is differentiable and $1$-convex. Indeed, for all $v\in V$, 
$$\mathcal{E}(v) = \frac{1}{2}\|v\|_V^2 + \left(\mathbb{E}\left[\left\langle f(T,.), v(T,.)\right\rangle _{H^{-1}(\mathcal{X}), H^1_0(\mathcal{X})}+ \frac{\rho}{2}\int_{\mathcal{X}}[g(T,x)-v(T,x)]_+^2dx\right]\right),$$
is the sum of a $1$-convex function ($V\ni v \mapsto \frac{1}{2}\|v\|_V^2$) and of a convex function ($V\ni v \mapsto \mathbb{E}\left[\left\langle f(T,.), v(T,.)\right\rangle _{H^{-1}(\mathcal{X}), H^1_0(\mathcal{X})}+ \frac{\rho}{2}\int_{\mathcal{X}}[g(T,x)-v(T,x)]_+^2dx\right]$).
 The functional $\mathcal{E}$ therefore obeys property (\ref{alpha}) with $\alpha = 1$. Hence, assumption $(A3)$ is satisfied.

Let us finally check that the gradient of $\mathcal{E}$ is Lipschitz.
For all $v,w,y\in V$,
\begin{eqnarray*}
\left| \langle \mathcal{E}'(v) - \mathcal{E}'(w), y \rangle \right| & \leq & \left| \mathbb{E}\left[ \int_{\mathcal{X}} \nabla_x(v(T,x)-w(T,x)).\nabla_xy(T,x) dx \right]\right| \\
 & &  +\rho\left|\mathbb{E}\left[ \int_{\mathcal{X}} ([g(T,x)-v(T,x)]_+ - [g(T,x) -w(T,x)]_+)y(T,x)dx \right]\right| \\
&\leq & \|v-w\|_V\|y\|_V \\
&&+ \rho \mathbb{E}\left[ \int_{\mathcal{X}}  \left|[g(T,x)-v(T,x)]_+ - [g(T,x) -w(T,x)]_+\right| \left|y(T,x)\right|dx \right].\\
\end{eqnarray*}

For $a,b\in\mathbb{R}$, it can easily be seen that $|[a]_+-[b]_+|\leq |a-b|$. This implies
\begin{eqnarray*}
\left| \langle \mathcal{E}'(v) - \mathcal{E}'(w), y\rangle \right| &  \leq & \|v-w\|_V\|y\|_V + \rho \mathbb{E}\left[ \int_{\mathcal{X}} |v(T,x)- w(T,x)| \;|y(T,x)|dx\right] \\
& \leq & \|v-w\|_V\|y\|_V\\
&& + \rho \left(\mathbb{E}\left[ \int_{\mathcal{X}} |v(T,x)- w(T,x)]|^2 dx\right]\right)^{1/2}\left(\mathbb{E}\left[ \int_{\mathcal{X}} |y(T,x)]|^2 dx\right]\right)^{1/2} .\\
\end{eqnarray*}

The Poincar\'e inegality in $H^1_0(\mathcal{X})$ implies that there exists a nonnegative constant $D\in\mathbb{R}_+$ such that for all $h\in V$,
$$\left|\mathbb{E}\left[ \int_{\mathcal{X}} |h(T,x)]|^2 dx\right]\right|^{1/2} \leq D\|h\|_V.$$ 

This yields
$$   \left| \langle \mathcal{E}'(v) - \mathcal{E}'(w), y \rangle \right| \leq (1+\rho D^2)\|v-w\|_V \|y\|_V,$$
hence,
$$\|\mathcal{E}'(v)-\mathcal{E}'(w)\|_V \leq (1+ \rho D^2)\|v-w\|_V.$$
The functional $\mathcal{E}$ then obeys property (\ref{Lipschitz}) 
with a constant $L= 1 + \rho D^2$ independent of the bounded set considered.

Thus, our obstacle problem (\ref{penalized}) falls into the general theoretical setting introduced in Section~2.1.

\vspace{0.5cm}

There exist several variants of the obstacle problem which could be tackled with our algorithm. We refer to Ref.~\cite{Gros07} or Ref.~\cite{Glow89} for such examples.

\subsubsection{High-dimensional Poisson equation}

\vspace{0.5cm}

Our algorithm may also be used to calculate the solution of other problems than obstacle problems. Other examples are high-dimensional nonlinear Poisson equations. 
A specific application where such high dimensional Poisson equations arise is the calculation of the so-called committor function
 in molecular dynamics\cite{Weinan}, which is an important quantity to compute reaction rates or to derive some effective dynamics for example.

Let $q\in\mathbb{N}^*$. The committor is the solution to the following problem: 
$$z = \mathop{\mbox{argmin}}_{ v \in W} \frac{1}{2}\int_{\mathbb{R}^q \setminus (\overline{A} \cup \overline{B})} |\nabla v (y)|^2 \exp(-U(y)) \, dy$$
where $q$ is typically large, $A$ and $B$ are disjoint smooth open sets of $\mathbb{R}^q$, $U : \mathbb{R}^q \to \mathbb{R}$ is a given potential function
 such that $\int_{\mathbb{R}^q} \exp(-U)< \infty$ and
 $$W= \left\{ 
\begin{array}{c}
 v\in L^2_{loc}(\mathbb{R}^q),\\
 \int_{\mathbb{R}^q  \setminus (\overline{A} \cup \overline{B})} |\nabla v (y)|^2\exp(-U(y)) \,dy< \infty, \\
 v=1 \text{ on $\overline{A}$} \text{ and }  v=0 \text{ on $\overline{B}$} \\
\end{array}
\right\}.$$
For $y \in \mathbb{R}^q \setminus (\overline{A} \cup \overline{B})$, $z(y)$ can be interpreted as the probability that the stochastic process~$Q_t^y$ solution to
$$Q_t^y = y - \int_0^t \nabla U (Q^y_s) \, ds + \sqrt{2} \,  W_t$$
reaches $\overline{A}$ before $\overline{B}$. Here, $W_t$ denotes a $q$-dimensional Brownian motion.

\vspace{0.2cm}
Let $p,d\in\mathbb{N}^*$ such that $q=p+d$. In this example, we consider the case when $C = \mathbb{R}^q \setminus \left (\overline{A} \cup \overline{B}\right)$ is bounded. 
Let $\mathcal{T}$ and $\mathcal{X}$ be open convex bounded subsets of $\mathbb{R}^p$ and $\mathbb{R}^d$ respectively such that 
$\overline{C} \subset \Omega := \mathcal{T} \times \mathcal{X}$ and such that $\mu\left( (A \cup B) \cap \Omega\right) \neq 0$ where $\mu$ denotes the Lebesgue measure. We also 
assume that $U \in \mathcal{C}^{\infty}(\mathbb{R}^q)$. In this case, the initial problem can be rewritten as a minimization problem set on
 $$\widetilde{W}= \left\{ 
v\in H^1(\mathcal{T} \times \mathcal{X}) | \, v=1 \text{ on $A\cap \Omega$} \text{ and }  v=0 \text{ on $B \cap \Omega$} \right\},$$
instead of $W$. Indeed, as $U\in\mathcal{C}^{\infty}(\mathbb{R}^q, \mathbb{R})$ and $\Omega$ is bounded, there exists constants $\gamma, \kappa >0$ such that for all $y\in\Omega$, 
$\gamma \leq \exp(- U(y)) \leq \kappa$. And thus, we have $v\in W$ if and only if $v|_{\Omega } \in \widetilde{W}$, $v|_{\overline{A} \setminus \Omega} = 1$ and $v|_{\overline{B} \setminus \Omega} = 0$.

The penalized version of the committor problem then reads
\begin{equation}
\label{poisson}
 u = \mathop{\mbox{argmin}}_{v\in H^1(\mathcal{T} \times \mathcal{X})} \mathcal{E}(v),
\end{equation}
where
$${\mathcal E}(v) = \frac{1}{2}\int_{\Omega} |\nabla v (y)|^2 \exp(-U(y)) \, dy + \frac{\rho}{2} \left(\int_{A \cap \Omega} |v(y) -1|^2 \, dy + \int_{B \cap \Omega} |v(y)|^2 \, dy\right),$$
for some $\rho >0$.

\vspace{0.5cm}

 Let us check that the general theoretical setting described in Section~2.1 is relevant for this problem.

In this case, we consider $V=H^1(\mathcal{T}\times \mathcal{X})$, $V_t = H^1(\mathcal{T})$ and $V_x = H^1(\mathcal{X})$. 
The inner products that are defined over these Hilbert spaces are the following. For all 
 $v_1, v_2\in V$, $r_1,r_2\in V_t$, $s_1, s_2\in V_x$,
 $$\langle v_1, v_2 \rangle_V = \int_{\mathcal{T}}\int_{\mathcal{X}} \left(\nabla v_1(t,x) . \nabla v_2(t,x) + v_1(t,x) v_2(t,x) \right) \,dt \,dx, $$
 $$\langle r_1, r_2\rangle_{V_t} = \int_{\mathcal{T}} \left(\nabla r_1(t). \nabla r_2(t) + r_1(t)r_2(t) \right) \,dt, $$
 $$\langle s_1, s_2\rangle_{V_x} = \int_{\mathcal{X}} \left( \nabla s_1(x). \nabla s_2(x)+ s_1(x)s_2(x) \right)\,dx.$$

\begin{rem}{Let us point out that in this case, $V\neq V_t \otimes V_x$. Indeed, for all $(r,s)\in V_t \times V_x$, the $V$-norm of the tensor product $r\otimes s$ reads
$$
 \|r\otimes s\|_V^2 =\|r\|_{L^2(\mathcal{T})}^2 \|\nabla s\|_{L^2(\mathcal{X})}^2 + \|\nabla r\|_{L^2(\mathcal{T})}^2 \|s\|_{L^2(\mathcal{X})}^2 + \|r\|_{L^2(\mathcal{T})}^2 \|s\|_{L^2(\mathcal{X})}^2 ,
$$
which is not a cross-norm, equivalent to the norm induced by $\|.\|_{V_t}$ and $\|.\|_{V_x}$ over $V_t\otimes V_x$, which is
$$\|r\otimes s \|_{V_t\otimes V_x} = \|r\|_{V_t}\|s\|_{V_x}.$$
Indeed, let us consider $\mathcal{T} = \mathcal{X} = (0,1)$, $r_l(t) = \frac{1}{l} \sin(l^2\pi t)$ and $s_l(x) = \frac{1}{l}\sin(l^2\pi x)$ 
for $(t,x)\in (0,1)^2$ and $l\in\mathbb{N}^*$. The sequence $\left(\|r_l\otimes s_l \|_V\right)_{n\in\mathbb{N}^*}$ is bounded, but the sequence 
$\left( \|r_l\|_{V_t} \|s_l\|_{V_x} \right)_{l\in\mathbb{N}^*}$ is not.
}
\end{rem}

 \vspace{0.2cm}
 Assumption $(A1)$ holds true, since $\widetilde{\Sigma} = \left\{ r \otimes s \; | \; (r,s)\in \mathcal{C}^{\infty}\left(\overline{\mathcal{T}}\right) \times \mathcal{C}^{\infty}\left(\overline{\mathcal{X}}\right) \right\}$ 
is such that $\widetilde{\Sigma} \subset \Sigma$ and $\mbox{Span}\left( \widetilde{\Sigma} \right)$ is dense in $H^1(\mathcal{T}\times \mathcal{X})$. Hence, $\mbox{Span}\left( \Sigma \right)$ is also dense 
 in $V$. 
 
 \vspace{0.2cm}

Let us prove that assumption $(A2)$ also holds true. If $((r_n, s_n))_{n\in\mathbb{N}} \in (V_t\times V_x)^{\mathbb{N}}$ is such that $(\|r_n\otimes s_n\|_V)_{n\in\mathbb{N}}$ is bounded, we 
 can extract a subsequence of $(r_n \otimes s_n)_{n\in\mathbb{N}^*}$ which weakly converges in $V$ towards an element $w\in V$. Besides, 
 there exists a nonnegative constant $C\in\mathbb{R}_+$ such that for all $n\in\mathbb{N}$,
 \begin{eqnarray*}
 \|r_n\otimes s_n\|_V^2 & = & \int_{\mathcal{T}\times \mathcal{X}} \left( |\nabla r_n(t)|^2|s_n(x)|^2 + |r_n(t)|^2|\nabla s_n(x)|^2 + |r_n(t)|^2 |s_n(x)|^2\right) \,dt \,dx \\
 & = & \|\nabla r_n\|_{L^2(\mathcal{T})}^2 \|s_n\|_{L^2(\mathcal{X})}^2 + \|r_n\|_{L^2(\mathcal{T})}^2 \|\nabla s_n\|_{L^2(\mathcal{X})}^2 + \|r_n\|_{L^2(\mathcal{T})}^2 \|s_n\|_{L^2(\mathcal{X})}^2 \\
 & \leq & C. \\
\end{eqnarray*}
 
 We can then choose $((r_n^*, s_n^*))_{n\in\mathbb{N}} \in (V_t\times V_x)^{\mathbb{N}}$ such that $r_n^* \otimes s_n^* = r_n \otimes s_n$ and such that $\|r_n^*\|_{L^2(\mathcal{T})} = 1$.
 The sequences $(r_n^*)_{n\in\mathbb{N}}$ and $(s_n^*)_{n\in\mathbb{N}}$ are then bounded in $L^2(\mathcal{T})$ and $H^1(\mathcal{X})$ 
 and we can extract subsequences which weakly converge in $L^2(\mathcal{T})$ and $H^1(\mathcal{X})$ respectively towards $r^{\infty}$ and $s^{\infty}$. As the weak convergences in $L^2(\mathcal{T})$ and $H^1(\mathcal{X})$ imply 
 the convergences in the distributional sense, $r_n^* \otimes s_n^* = r_n \otimes s_n$ necessarily converges towards $r^{\infty}\otimes s^{\infty}$ in the distributional sense by Lemma 2.1. 
 As the weak convergence in $V$ also implies the convergence in the sense of the distributions, by uniqueness of the limit, $w=r^{\infty} \otimes s^{\infty}$.
 Let us suppose $w\neq 0$. In that case, we have $r^{\infty} \neq 0$ and $s^{\infty} \neq 0$. Besides, we have
 $$\|w\|_V^2 = \|r^{\infty}\|_{L^2(\mathcal{T})}^2 \| \nabla s^{\infty} \|_{L^2(\mathcal{X})}^2 + \|r^{\infty}\|_{H^1(\mathcal{T})}^2 \| s^{\infty} \|^2_{L^2(\mathcal{X})},$$
 hence
 $$ \|r^{\infty}\|_{H^1(\mathcal{T})}^2 \leq \frac{\|w\|_V^2 }{\| s^{\infty} \|_{L^2(\mathcal{X})}^2}.$$
 As a consequence $\|r^{\infty}\|_{H^1(\mathcal{T})}$ is finite and $r^{\infty} \in H^1(\mathcal{T})$. Hence $w = r^{\infty} \otimes s^{\infty} \in \Sigma$. If $w=0$, then obviously $w\in\Sigma$.
 Hence, assumption $(A2)$ holds true.
 
 \vspace{0.2cm}

The functional $\mathcal{E}$ is differentiable and strongly convex. To prove this, it is sufficient to prove that there exists a constant $\alpha \in \mathbb{R}_+^*$ 
such that for all $v,w\in V$, $\langle \mathcal{E}'(v) - \mathcal{E}'(w), v-w \rangle \geq \alpha \|v-w\|_V^2$. Indeed, there exists $\gamma>0$ such that 
for all $y\in\mathbb{R}^q$, $\exp (-U(y)) \geq \gamma$. Thus, there exists a constant $\delta >0$ such that,
for all $v,w\in V$, 
 $$\langle \mathcal{E}'(v) - \mathcal{E}'(w), v-w\rangle \geq  \delta \left( \int_{\Omega} |\nabla (v-w)|^2 + \int_{A \cap \Omega} |v-w|^2 + \int_{B \cap \Omega} |v-w|^2\right).$$
To prove that the functional $\mathcal{E}$ is strongly convex, it is sufficient to have the following inequality: there exists 
a constant $C_{\Omega}\in \mathbb{R}_+^*$ such that for all $v\in H^1(\Omega)$, 
\begin{equation}
 \label{Poinc}
\int_{\Omega} |\nabla v|^2 + \int_{(A \cup B) \cap \Omega} |v|^2 \geq C_{\Omega}\|v\|_{H^1(\Omega)}^2.
\end{equation}
As $\mathcal{T}$ and $\mathcal{X}$ are bounded open convex subsets of $\mathbb{R}^p$ and $\mathbb{R}^d$ respectively, 
$\Omega$ is then a bounded open convex subset of $\mathbb{R}^q$ such that $\mu( (A\cup B) \cap \Omega) \neq 0$ and inequality (\ref{Poinc}) is a well-known 
Poincare-like inequality. 

Hence, assumption $(A3)$ is satisfied.
 
 \vspace{0.2cm}

Let us check that the gradient of $\mathcal{E}$ is Lipschitz.
For all $v,w,z\in V$,
 \begin{eqnarray*}
 \left| \langle \mathcal{E}'(v) - \mathcal{E}'(w), z \rangle \right| & \leq & \left|  \int_{\Omega} \nabla(v(t,x)-w(t,x)).\nabla z(t,x) \exp(-U(t,x)) \,dt\,dx \right| \\
 &&  + \rho\left|  \int_{(A\cup B) \cap \Omega}(v(t,x)-w(t,x))z(t,x) \,dt\,dx \right| \\
&\leq & \|\exp(-U)\|_{L^{\infty}(\Omega)}\|\nabla(v-w)\|_{L^2(\Omega)}\|\nabla z\|_{L^2(\Omega)}\\
&& + \rho \|v-w\|_{L^2((A\cup B) \cap \Omega)}\|z\|_{L^2((A\cup B) \cap \Omega)}\\
& \leq &  \|\exp(-U)\|_{L^{\infty}(\Omega)}\|\nabla(v-w)\|_{L^2(\Omega)}\|\nabla z\|_{L^2(\Omega)}\\
&& + \rho \|v-w\|_{L^2( \Omega)}\|z\|_{L^2( \Omega)}\\
 & \leq & \|v-w\|_V\|z\|_V(\|\exp(-U)\|_{L^{\infty}(\Omega)}+\rho).\\
 \end{eqnarray*}
Hence
 $$\|\mathcal{E}'(v)-\mathcal{E}'(w)\|_V \leq (\|\exp(-U)\|_{L^{\infty}(\Omega)}+\rho)\|v-w\|_V.$$
  The functional $\mathcal{E}$ therefore obeys property (\ref{Lipschitz}) 
 with a constant $L = \|\exp(-U)\|_{L^{\infty}(\Omega)}+\rho$ independent of the bounded subset considered.
 
 Thus, the committor problem falls into the general theoretical setting introduced in Section~2.1.

\section{Proof of Theorem 2.1} 

\subsection{The iterations are well-defined}

We begin by proving that the iterations of the algorithm are well-defined. For this, we will need the following lemma.

\begin{lem}
 
Let $w$ be a function in $V$. Then there exists a pair $(r,s)\in V_t \times V_x$ such that $\mathcal{E}(w +r\otimes s) < \mathcal{E}(w)$ if and only if $\mathcal{E}'(w) \neq 0$.

\end{lem}

\begin{proof}
 
Let $w\in V$  and let us suppose that  $\mathcal{E}(w+r\otimes s) \geq \mathcal{E}(w)$ for all $(r,s) \in V_t \times V_x$. For a given pair $(r,s)$, for all $\varepsilon \in\mathbb{R}$,
$$\mathcal{E}( w + \varepsilon r\otimes s)- \mathcal{E}(w) \geq 0.$$

As a consequence, we have the following by letting $\varepsilon$ go to $0$: $\langle\mathcal{E}'(w),r\otimes s\rangle = 0$. 
This holds for all $(r,s) \in V_t \times V_x$. Hence, for all $z\in \mbox{Span}(\Sigma)$, we also have $\langle\mathcal{E}'(w),z \rangle = 0$, and the density of 
$\mbox{Span}(\Sigma)$ in $V$, which is ensured by assumption $(A1)$, yields
$$\mathcal{E}'(w) =0.$$.

\vspace{0.1cm}

Conversely, let us assume that $ \mathcal{E}'(w) = 0$. Then, as $\mathcal{E}$ is $\alpha$-convex, $w$ is necessarily the global minimizer of $\mathcal{E}$ and, 
in particular, we have for all $(r,s)  \in V_t \times V_x$,
$$\mathcal{E}(w+r\otimes s) \geq \mathcal{E}(w).$$

This concludes the proof.
\end{proof}

Using this lemma, the following result can be proved:

\begin{prop} 
For all $n\in \mathbb{N}^*$, there exists a solution $(r_n, s_n) \in V_t \times V_x$ to the minimization problem (\ref{algorithm}) .

 Moreover, $r_n \otimes s_n \neq 0$ if and only if $u_{n-1} \neq u$, where $u_n$ is defined by (\ref{undef}).

\end{prop}

\begin{proof}

Firstly, let us prove the existence of a minimizer for problem (\ref{algorithm}). 

Let $n \in \mathbb{N}^*$. For all $(r,s)\in  V_t \times V_x$ ,  $\mathcal{E}(u_{n-1}+r\otimes s)\geq \mathcal{E}(u)$. So 
$m = \mathop{\inf}_{(r,s) \in  V_t \times V_x} \mathcal{E}(u_{n-1}+r\otimes s)$ exists in $\mathbb{R}$.

We then consider a minimizing sequence $(r^{(l)}, s^{(l)})_{l\in\mathbb{N}}\in\left(V_t\times V_x\right)^{\mathbb{N}}$ such that
$$\mathop{\lim}_{l\rightarrow \infty}\mathcal{E}(u_{n-1}+r^{(l)}\otimes s^{(l)}) = m.$$

Using (\ref{alpha}) and the fact that $\mathcal{E}'(u)=0$, we have
$$\mathcal{E}(u_{n-1} + r^{(l)} \otimes s^{(l)}) - \mathcal{E}(u) \geq \frac{\alpha}{2} ||u_{n-1} + r^{(l)}\otimes s^{(l)} -u||_V^2 .$$

Then the sequence $(r^{(l)}\otimes s^{(l)})_{l\in\mathbb{N}}$ is bounded in $V$ because $(\mathcal{E}(u_{n-1}+r^{(l)}\otimes s^{(l)}))_{l\in\mathbb{N}}$ is convergent and consequently bounded.

As assumption $(A2)$ is satisfied, we can then extract a subsequence (which we still denote $(r^{(l)}\otimes s^{(l)})_{l\in\mathbb{N}}$) which weakly converges in $V$ towards an element of $\Sigma$. In other words, 
there exist $r^{\infty} \in V_t$ and $s^{\infty} \in V_x$ such that $(r^{(l)}\otimes s^{(l)})_{l\in\mathbb{N}}$ weakly converges in $V$ towards $r^{\infty} \otimes s^{\infty}$. 

Furthermore, as the functional $\mathcal{E}$ is convex and continuous on $V$,
$$\mathcal{E}(u_{n-1}+ r^{\infty}\otimes s^{\infty}) \leq \mathop{\lim}_{l\rightarrow\infty} \mathcal{E}(u_{n-1} + r^{(l)}\otimes s^{(l)}) = m.$$

Hence  $\mathcal{E}(u_{n-1}+ r^{\infty}\otimes s^{\infty}) = m$ so that $(r^{\infty}, s^{\infty})$ is a minimizer of problem (\ref{algorithm}).

\vspace{0.1cm}

Let us prove now that  $r^{\infty}\otimes s^{\infty}\neq 0$ if and only if $u_{n-1}\neq u$.

If $u_{n-1} = u$, we have $\mathcal{E}(u+r\otimes s) > \mathcal{E}(u)$ for all $(r,s)\in V_t \times V_x$ such that $r\otimes s\neq 0$
 as $\mathcal{E}$ is strictly convex. So a minimizer $r^{\infty}\otimes s^{\infty}$ of problem (\ref{algorithm}) must necessarily satisfy $r^{\infty}\otimes s^{\infty} = 0$.

Conversely, if $u_{n-1} \neq u$, we have $ \mathcal{E}'(u_{n-1}) \neq 0$ and from Lemma 3.1, there exists a pair $(r,s)\in V_t \times V_x$ such that $\mathcal{E}(u_{n-1} + r \otimes s) < \mathcal{E}(u_{n-1})$.
 Hence, $\mathcal{E}(u_{n-1}+r^{\infty}\otimes s^{\infty}) < \mathcal{E}(u_{n-1})$ and $r^{\infty}\otimes s^{\infty}$ cannot be equal to $0$.
\end{proof}

\begin{prop}
For each $n\in\mathbb{N}^*$, a minimizer $(r_n, s_n)$ of problem (\ref{algorithm}) obeys the following Euler equation:
\begin{equation}
\label{EL}
\forall (r,s) \in V_t\times V_x, \; \; \; \left\langle\mathcal{E}'(u_n), r \otimes s_n + r_n \otimes s\right\rangle = 0.
\end{equation}
\end{prop}

This result is obtained by considering the first-order conditions of the minimization problem (\ref{algorithm}). This will be useful in the proof of convergence.

\subsection{Proof of convergence}

In this subsection, we present the different steps of the proof.

\begin{lem}
The series $\sum_{n=1}^{\infty} \|r_n \otimes s_n\|^2_V$ and the sequence $\left(\mathcal{E}(u_n)\right)_{n\in\mathbb{N}^*}$ are convergent.
\end{lem}

\begin{proof}
 
Let us set $E_n = \mathcal{E}(u_n) = \mathcal{E}\left(\sum_{k=1}^n r_k\otimes s_k\right).$

Using (\ref{algorithm}), $E_n \leq \mathcal{E}\left(u_{n-1} + r\otimes s\right)$ for all $(r,s)\in V_t\times V_x$, and in particular, by taking $r\otimes s=0$, 
$(E_n)_{n\in\mathbb{N}^*}$ is a non-increasing sequence. Moreover, it is bounded from below. Indeed, for all $n\in \mathbb{N}^*$, we have $E_n \geq \mathcal{E}(u)$. Thus, it is convergent.

This implies that the sequence defined as $W_n = E_{n-1} - E_n$ is  nonnegative, converges to $0$, and satisfies $\sum_{n=1}^{\infty} W_n < + \infty$.

Besides, the $\alpha$-convexity of $\mathcal{E}$ yields the following inequality: 

$W_n \geq -\langle\mathcal{E}'(u_n), r_n \otimes s_n\rangle + \frac{\alpha}{2} \|r_n\otimes s_n\|_V^2.$

Using the Euler equations (\ref{EL}), $\langle\mathcal{E}'(u_n), r_n\otimes s_n\rangle=0,$
and thus, $W_n \geq \frac{\alpha}{2} \|r_n\otimes s_n\|_V^2 $. Hence the result.
\end{proof}

\begin{lem}
 
The sequence $\left(u_n\right)_{n\in\mathbb{N}^*}$ is bounded in $V$.

\end{lem}

\begin{proof}
 
By $\alpha$-convexity of the functional $\mathcal{E}$, we have
\begin{eqnarray*}
\mathcal{E}(0) & \geq & \mathcal{E}(u_n) \geq \mathcal{E}(u) + \langle\mathcal{E}'(u) , u_n -u\rangle\\
& & + \frac{\alpha}{2}\|u-u_n\|_V^2.\\
\end{eqnarray*}
Thus $\|u-u_n\|_V^2 \leq \frac{2}{\alpha}(\mathcal{E}(0) - \mathcal{E}(u)).$

Therefore, the sequence $(u_n)_{n\in\mathbb{N}^*}$ is bounded in $V$.
\end{proof} 

The following estimate is essential for the proof of convergence.

\begin{prop}
 
There exists a constant $A\in\mathbb{R}_+$ such that, for all $n\in\mathbb{N}^*$ and all $(r,s)\in V_t\times V_x$,
\begin{equation}
\label{estimate}
\left| \left\langle \mathcal{E}'(u_{n-1}), r\otimes s\right\rangle \right| \leq A\| r_n\otimes s_n\|_V \|r\otimes s\|_V.
\end{equation}

\end{prop}

\begin{proof}
Let $M\in\mathbb{R}_+$ be such that for all $n\in\mathbb{N}^*$, $\|u_n\|_V \leq M$. Its existence is ensured by Lemma~3.3. Let $N\in\mathbb{R}_+$ be such that for all $n\in\mathbb{N}^*$, $\|r_n\otimes s_n\|_V \leq N$. 
Let $K=\overline{B}(0, M+2N+3)$ be the closed ball of~$V$ centered at $0$ 
of radius $M+2N+3$. Let $L$ be the Lipschitz constant associated with~$K$ in (\ref{Lipschitz}).

For all $(r,s)\in V_t\times V_x$, we have $\mathcal{E}(u_{n-1}+r\otimes s) - \mathcal{E}(u_{n-1}+r_n \otimes s_n) \geq 0.$

Then, by the convexity of $\mathcal{E}$, we have the following inequality
$$\left\langle\mathcal{E}' (u_{n-1}+r\otimes s), r_n\otimes s_n - r\otimes s\right\rangle \leq  \mathcal{E}(u_{n-1}+r_n\otimes s_n) -\mathcal{E}(u_{n-1}+r\otimes s) \leq 0,$$
which leads to
\begin{equation}
\label{ineg}
\langle\mathcal{E}' (u_{n-1}+r\otimes s),  r\otimes s\rangle \geq \langle\mathcal{E}' (u_{n-1}+r\otimes s),  r_n\otimes s_n\rangle.
\end{equation}

Let $(r,s)\in V_t\times V_x$ such that $\|r\otimes s \|_V \leq \mbox{max}\left(1, \|r_n\otimes s_n\|_V\right)$. We then have, by using (\ref{Lipschitz}) and (\ref{ineg}),
\begin{eqnarray*}
-\langle \mathcal{E}'(u_{n-1}), r\otimes s \rangle & = & -\langle \mathcal{E}'(u_{n-1}), r\otimes s \rangle + \langle \mathcal{E}'(u_{n-1}+ r\otimes s), r\otimes s \rangle\\
& & - \langle \mathcal{E}'(u_{n-1}+ r\otimes s), r\otimes s \rangle \\
& \leq &  L \|r\otimes s \|_V^2 - \langle \mathcal{E}'(u_{n-1}+ r\otimes s), r\otimes s \rangle \\
& = &  L \|r\otimes s \|_V^2 - \langle \mathcal{E}'(u_{n-1}+ r\otimes s), r\otimes s \rangle + \langle \mathcal{E}'(u_{n-1}+ r\otimes s), r_n\otimes s_n \rangle \\
&& - \langle \mathcal{E}'(u_{n-1}+ r\otimes s), r_n\otimes s_n \rangle \\
& \leq & L \|r\otimes s \|_V^2 - \langle \mathcal{E}'(u_{n-1}+ r\otimes s), r_n\otimes s_n \rangle \\
& = & L \|r\otimes s \|_V^2 - \langle \mathcal{E}'(u_{n-1}+ r\otimes s), r_n\otimes s_n \rangle + \langle \mathcal{E}'(u_{n-1}+ r_n\otimes s_n), r_n\otimes s_n \rangle \\
&& - \langle \mathcal{E}'(u_{n-1}+ r_n\otimes s_n), r_n\otimes s_n \rangle \\
& \leq & L \|r\otimes s \|_V^2 +L \|r\otimes s - r_n \otimes s_n\|_V \|r_n\otimes s_n\|_V \\
& & - \langle \mathcal{E}'(u_{n-1}+ r_n\otimes s_n), r_n\otimes s_n \rangle \\
& = & L \|r\otimes s \|_V^2 +L \|r\otimes s - r_n \otimes s_n\|_V \|r_n\otimes s_n\|_V .\\
\end{eqnarray*}

The last line has been obtained by taking into account the fact that $\langle \mathcal{E}'(u_{n-1}+ r_n\otimes s_n), r_n\otimes s_n \rangle = 0$ because of the Euler equation
 (\ref{EL}).

Thus, for all $(r,s) \in V_t\times V_x$ such that $\|r\otimes s\|_V \leq \mbox{max}\left(1, \|r_n \otimes s_n\|_V\right)$, 
$$\langle\mathcal{E}'(u_{n-1}), r\otimes s\rangle + L\|r\otimes s\|_V^2 + L\|r\otimes s\|_V\|r_n\otimes s_n\|_V + L\|r_n\otimes s_n\|_V^2 \geq 0 .$$

As a consequence,
$$|\langle\mathcal{E}'(u_{n-1}), r\otimes s\rangle| \leq L\|r\otimes s\|_V^2 + L\|r\otimes s\|_V\|r_n\otimes s_n\|_V +L \|r_n\otimes s_n\|_V^2.$$

Let $(r,s) \in V_t \times V_x$ such that $\| r \otimes s \|_V = 1$ and $t\in\mathbb{R}$ such that $t\leq \mbox{max}\left(1, \|r_n\otimes s_n\|_V\right)$.
 Then, we have
$$|\langle\mathcal{E}'(u_{n-1}), t r\otimes s\rangle| \leq Lt^2\|r\otimes s\|_V^2 + Lt\|r\otimes s\|_V\|r_n\otimes s_n\|_V +L \|r_n\otimes s_n\|_V^2.$$

And, by setting $t=\| r_n \otimes  s_n\|_V$, we obtain the following inequality for all $(r,s)\in V_t \times V_x$ such that $\|r\otimes s\|_V = 1$,
$$|\langle\mathcal{E}'(u_{n-1}), r\otimes s\rangle| \leq 3L \|r_n\otimes s_n\|_V\|r\otimes s\|_V.$$

Of course, this inequality also holds true for all $(r,s)\in V_t\times V_x$ such that $\|r\otimes s\|_V \neq 1$. Therefore, (\ref{estimate}) holds with $A=3L$.
\end{proof}

We now state an elementary result which will be useful in the sequel.

\begin{lem}
Let $(a_n)_{n\in\mathbb{N}^*}$ be a sommable sequence of $\mathbb{R}_+$. Then, there exists a subsequence of $(na_n)_{n\in\mathbb{N}^*}$ which converges to $0$.  
\end{lem}

\begin{proof}
If such a subsequence could not be extracted, it would imply
$$\exists \varepsilon_0 >0 \; , \; \exists n_0 \in \mathbb{N}^* \; , \; \forall n \geq n_0\; , \; na_n \geq  \varepsilon_0.$$

Thus, the series $\sum_{n=1}^{\infty} a_n$ would diverge. Hence the contradiction.
\end{proof}

We are now in position to complete the proof of Theorem 2.1.

\begin{proof}

By Lemma 3.2, the sequence $(\mathcal{E}(u_n))_{n\in\mathbb{N}^*}$ is convergent. Let us denot its limit by $E$. We want to prove that $E=\mathcal{E}(u)$.

Firstly, for all $n\in \mathbb{N}^*$, $\mathcal{E}(u_n) \geq \mathcal{E}(u)$,
 since $u$ is the global minimizer of the functional $\mathcal{E}$. By letting $n$ go to infinity, we obtain $E \geq \mathcal{E}(u).$

It remains to prove that $E\leq \mathcal{E}(u)$.

Let us first prove that $\left(\mathcal{E}'(u_n)\right)_{n\in\mathbb{N}^*}$ weakly converges to $0$ in $V$. Let $M\in\mathbb{R}_+$ such that for all $n\in\mathbb{N}^*$, $\|u_n\|_V \leq M$.
 Its existence is ensured by Lemma~3.3. 
Let $K=\overline{B}(0, M+2+\|u\|_V)$ be the closed ball of~$V$ centered at $0$ of radius $M+2+\|u\|_V$. Let $L$ be the Lipschitz constant associated with~$K$ in (\ref{Lipschitz}). 
Using (\ref{Lipschitz}) and the fact that $\mathcal{E}'(u) = 0$, we have
 $\|\mathcal{E}'(u_n)\|_V\leq  L\|u-u_n\|_V$
 and as $(u_n)_{n\in\mathbb{N}^*}$ is bounded in $V$ by Lemma 3.3, we deduce 
that $(\mathcal{E}'(u_n))_{n\in\mathbb{N}^*}$ is also bounded in $V$. We can then extract a subsequence of $(\mathcal{E}'(u_n))_{n\in\mathbb{N}^*}$ which 
weakly converges in $V$ towards $w\in V$. 
By using Proposition 3.3 and by letting $n$ go to infinity in (\ref{estimate}), we deduce that $\langle w, r\otimes s\rangle = 0$ for all $(r,s)\in V_t\times V_x$.
 Then, as $\mbox{Span}(\Sigma)$ is dense in $V$
 with assumption $(A1)$, necessarily $w=0$. Thus the sequence $(\mathcal{E}'(u_n))_{n\in\mathbb{N}^*}$ weakly converges to $0$ in $V$.

As $\mathcal{E}$ is convex, we have the following inequality for all $n\in\mathbb{N}^*$,
\begin{equation}
\label{fin}
\mathcal{E}(u_n) \leq \mathcal{E}(u) + \langle\mathcal{E}'(u_n), u_n -u \rangle_V.
\end{equation}

Let us prove that we can extract a subsequence of $(\langle\mathcal{E}'(u_n), u_n\rangle)_{n\in\mathbb{N}^*}$ which converges to $0$. Let $n\in \mathbb{N}^*$.
By using Proposition 3.3,
\begin{eqnarray*}
|\langle\mathcal{E}'(u_n), u_n\rangle| & \leq &  \sum_{k=1}^n |\langle\mathcal{E}'(u_n), r_k\otimes s_k\rangle|,\\
  & \leq & A \sum_{k=1}^n \|r_{n+1}\otimes s_{n+1}\|_V  \|r_k\otimes s_k\|_V ,\\
 & \leq & A(n \|r_{n+1}\otimes s_{n+1}\|_V^2)^{1/2} \left(\sum_{k=1}^{n} \|r_k\otimes s_k\|_V^2\right)^{1/2}.\\
\end{eqnarray*}

As the sequence $\left(\sum_{k=1}^n \|r_k\otimes s_k\|_V^2\right)_{n\in\mathbb{N}^*}$ converges by Lemma 3.2, we have $\sum_{k=1}^{n} \|r_k\otimes s_k\|_V^2 \leq \sum_{k=1}^{\infty} \|r_k\otimes s_k\|_V^2 < \infty$. 
Furthermore, we can also extract a subsequence from $(n\|r_{n+1}\otimes s_{n+1} \|_V^2)_{n\in\mathbb{N}^*}$ which converges to $0$ (see Lemma 3.4). 

We can then extract a subsequence from $(\langle \mathcal{E}'(u_n), u_n\rangle_V)_{n\in\mathbb{N}^*}$ which converges to $0$.

By letting $n$ go to infinity in (\ref{fin}) with this subsequence, we obtain that $E\leq \mathcal{E}(u)$.

We have thus proved that $E = \mathcal{E}(u)$.

\vspace{0.5cm}

Besides, as the functional $\mathcal{E}$ is $\alpha$-convex, (\ref{alpha}) yields the following inequality,
$$\frac{\alpha}{2} \|u-u_n\|_V^2 \leq \mathcal{E}(u_n) - \mathcal{E}(u),$$
which necessarily implies that $\|u-u_n\|_V$ converges to $0$ when $n$ goes to infinity, which proves that $(u_n)_{n\in\mathbb{N}^*}$ strongly converges towards $u$ in $V$.
\end{proof} 

\section{Rate of convergence in the finite-dimensional case}

In the case when $V_t$ and $V_x$ are finite-dimensional, we are able to prove that the algorithm converges exponentially fast.  

\begin{thm}
 We assume that $V_t$ and $V_x$ are finite-dimensional and that assumptions $(A1)$, $(A2)$, $(A3)$ and $(A4)$ are fulfilled.
 Then there exist two constants $\tau>0$ and $\sigma \in (0,1)$ such that for all $n\in\mathbb{N}^*$,
\begin{equation}
\label{exponential}
0\leq \mathcal{E}(u_n) -\mathcal{E}(u) \leq \tau \sigma^n,
\end{equation}
and
\begin{equation}
 \label{vector}
\|u-u_n\|_V \leq \sqrt{\frac{2\tau}{ \alpha}} \sigma^{n/2}.
\end{equation}\label{1}

\end{thm}

\begin{proof}
Let us denote by $l = \mbox{dim} V_t$ and $m = \mbox{dim}V_x$. Then we can consider that $V_t = \mathbb{R}^l$, $V_x = \mathbb{R}^m$ and $V = \mathbb{R}^{l\times m}$ (which is implied by $(A1)$).

As the spaces are finite-dimensional, all the norms are equivalent, and we can consider without loss of generality that $\|.\|_{V_t}$, $\|.\|_{V_x}$ and $\|.\|_V$ are equal 
to the Frobenius norms of $\mathbb{R}^l$, $\mathbb{R}^m$ and $\mathbb{R}^{l\times m}$ defined by:
\begin{equation}
\label{normkl}
 \begin{array}{ccc}
  \|R\|_l^2 &= &R^T R,\\[5pt]
\|S\|_m^2 & = & S^T S,\\ [5pt]
\|U\|_{lm}^2 & =&  \mbox{Tr}(U^T U).\\
 \end{array}
\end{equation}
Notice that for all $(R,S)\in\mathbb{R}^l \times \mathbb{R}^m$, 
$$\|R\otimes S \|_V  = \|RS^T\|_{lm} = \|R\|_l\|S\|_m.$$

Let $(\phi_i)_{1\leq i \leq l}$ and $(\psi_j)_{1\leq j\leq m}$ be orthonormal bases of $V_t$ and $V_x$ respectively. Then, $(\phi_i\otimes \psi_j)_{1\leq i \leq l, 1\leq j\leq m}$
 forms an orthonormal basis of $V$.

Our goal is to prove that there exists a constant $\sigma \in (0,1)$ such that for all $n\in\mathbb{N}^*$,
\begin{equation}
\label{finalin}
 \mathcal{E}(u_n) - \mathcal{E}(u) \leq \sigma \left( \mathcal{E}(u_{n-1}) - \mathcal{E}(u) \right).
\end{equation}

Let $n\in\mathbb{N}^*$. Let us notice that
\begin{equation}
\label{ratebeg}
 \mathcal{E}(u_n) - \mathcal{E}(u) = \mathcal{E}(u_n) - \mathcal{E}(u_{n-1}) + \mathcal{E}(u_{n-1}) - \mathcal{E}(u).
\end{equation}

As for all $n\in\mathbb{N}^*$, $\mathcal{E}(u_n) - \mathcal{E}(u)\geq 0$, it is then sufficient with (\ref{ratebeg}) to prove that there exists $\lambda \in (0,1)$ such that
\begin{equation}
\label{suite}
\mathcal{E}(u_n) - \mathcal{E}(u_{n-1}) \leq - \lambda \left( \mathcal{E}(u_{n-1}) - \mathcal{E}(u) \right),
\end{equation}
to have (\ref{finalin}) with $\sigma = 1 -\lambda \in (0,1)$.

\vspace{0.5cm}

Let us notice that (\ref{alpha}) and (\ref{EL}) yield
\begin{equation}
\label{init}
 \mathcal{E}(u_n) - \mathcal{E}(u_{n-1}) \leq -\frac{\alpha}{2}\|r_n\otimes s_n\|_V^2.
\end{equation}

Besides, let $M \in \mathbb{R}_+$ such that for all $n\in\mathbb{N}^*$, $\|u_n\|_V \leq M$. Its existence is ensured by Lemma~3.3. Let $K=\overline{B}(0, M+\|u\|_V+2)$ be the closed ball of $V$ centered at $0$ 
of radius $M+\|u\|_V+2$. Let $L$ be the Lipschitz constant of the gradient of $\mathcal{E}$ associated to $K$
 in (\ref{Lipschitz}).

Using (\ref{Lipschitz}) and the fact that $\mathcal{E}'(u) = 0$, we also have,
\begin{equation}
 \label{un2}
 \mathcal{E}(u_{n-1}) - \mathcal{E}(u) \leq L\|u-u_{n-1}\|_V^2.
\end{equation}

With (\ref{init}) and (\ref{un2}), it is sufficient to prove that there exists a constant $\kappa\in(0,1)$ such
 that for all $n\in\mathbb{N}^*$, 
\begin{equation}
\label{kappa}
 \|r_n\otimes s_n\|_V \geq \kappa \|u - u_{n-1}\|_V,
\end{equation}
in order to have (\ref{suite}) and hence (\ref{finalin}).

Indeed, if (\ref{kappa}) holds, we then have, using (\ref{init}), (\ref{kappa}) and (\ref{un2}),
\begin{eqnarray*}
  \mathcal{E}(u_n) - \mathcal{E}(u_{n-1}) &\leq &-\frac{\alpha}{2}\|r_n\otimes s_n\|_V^2\\
& \leq & - \frac{\alpha}{2} \kappa^2 \|u - u_{n-1}\|_V^2 \\
& \leq & -\frac{\alpha}{2L} \kappa^2 \left( \mathcal{E}(u_{n-1}) - \mathcal{E}(u) \right).\\
\end{eqnarray*}

As the $\alpha$-convexity of $\mathcal{E}$ and the fact that $\mathcal{E}'(u) = 0$ yields
\begin{equation}
\label{un1}
 \mathcal{E}(u_{n-1}) - \mathcal{E}(u) \geq \frac{\alpha}{2}\|u-u_{n-1}\|_V^2,
\end{equation}
inequalities (\ref{un1}) and (\ref{un2}) then imply that $\frac{\alpha}{2} \leq L$ and then (\ref{suite}) holds with $\lambda = \frac{\alpha}{2L} \kappa^2 \in (0,1)$.

\vspace{0.5cm}

Let us prove inequality (\ref{kappa}). 
From Proposition~3.3, estimate (\ref{estimate}) holds true. As 
$(\phi_i\otimes \psi_j)_{1\leq i \leq l, 1\leq j\leq m}$ forms an orthonormal
 basis of $V$, we obtain, using (\ref{estimate}),
\begin{eqnarray*}
\|\mathcal{E}'(u_{n})\|_V^2 & = & \sum_{i=1}^l \sum_{j=1}^m \langle \mathcal{E}'(u_{n}) , \phi_i\otimes \psi_j\rangle^2 \\
& \leq & \sum_{i=1}^l \sum_{j=1}^m A^2 \|r_{n+1}\otimes s_{n+1}\|_V^2 \|\phi_i\otimes \psi_j\|_V^2 \\
&  = &  lmA^2 \|r_{n+1} \otimes s_{n+1}\|_V^2.
\end{eqnarray*}

We then have the following estimate:
\begin{equation}
 \label{normeprime}
\|\mathcal{E}'(u_n)\|_V \leq \sqrt{lm}A\|r_{n+1} \otimes s_{n+1}\|_V.
\end{equation}

The $\alpha$-convexity of $\mathcal{E}$ and estimate (\ref{normeprime}) lead to
\begin{eqnarray*}
 \mathcal{E}(u_{n-1}) - \mathcal{E}(u) &\leq& -\langle \mathcal{E}'(u_{n-1}), u-u_{n-1} \rangle - \frac{\alpha}{2} \|u-u_{n-1}\|_V^2\\
& \leq & \sqrt{lm} A\|r_{n} \otimes s_{n}\|_V \|u-u_{n-1}\|_V - \frac{\alpha}{2} \|u-u_{n-1}\|_V^2.\\
\end{eqnarray*}
Besides, by using the fact that $\mathcal{E}(u_{n-1}) - \mathcal{E}(u) \geq 0$, we obtain
$$\|r_n\otimes s_n\|_V \geq \frac{\alpha}{2\sqrt{lm}A}\|u-u_{n-1}\|_V,$$
which is (\ref{kappa}) with $\kappa = \frac{\alpha}{2\sqrt{lm}A} \in (0,1)$ for $A$ large enough.

Hence the result.
\end{proof}

\begin{rem}
 { This result can be generalized to the case of tensor products of more than two Hilbert spaces. Indeed, with the notation of Remark~2.2,  and if we denote
$l_1 = \mbox{dim} V_1, \cdots, l_q = \mbox{dim} V_q$, estimate (\ref{normeprime}) becomes
$$\|\mathcal{E}'(u_n)\|_V \leq \sqrt{l_1 \cdots l_q}A\|r_{n+1} \otimes s_{n+1}\|_V,$$
and the proof still holds.}
\end{rem}

\section{Case of a local minimum}

We are able to extend the results of Theorem 2.1 and Theorem 4.1 in the case when $(r_n,s_n)$ in (\ref{algorithm}) is only defined as a \bfseries local \normalfont 
minimum which ensures the decrease of the energy, more precisely, 
when $(r_n, s_n)$ is defined recursively as:
\begin{equation}
 \label{localalgo}
 (r_n,s_n) = \mathop{\mbox{local argmin}}_{(r,s)\in V_t\times V_x} \mathcal{E}\left(u_{n-1}+r\otimes s\right),
\end{equation}
such that
\begin{equation}
\label{decrease}
 \mathcal{E}\left(u_n\right) < \mathcal{E}\left(u_{n-1}\right),
\end{equation}
where $u_n$ is defined as in (\ref{undef}). 

To extend
these results, we will need an additional assumption (which is naturally fulfilled in the
finite dimensional case), see Remark 5.2 below:
\begin{description}
\item{ $(A5)$ } There exist $\beta, \gamma \in \mathbb{R}_+$ such that
\begin{equation}
 \label{SVD}
\forall (r,s)\in V_t\times V_x,\; \;\beta \|r\|_{V_t} \|s\|_{V_x} \leq \|r\otimes s \|_V \leq \gamma \|r\|_{V_t} \|s\|_{V_x}.
\end{equation}  
\end{description}

\begin{thm}
Let us suppose that the assumptions $(A1)$, $(A2)$, $(A3)$, $(A4)$ and $(A5)$ hold true. Then, the iterations of the algorithm described above are well-defined in the sense that 
(\ref{localalgo}) has at least one local minimizer $(r_n, s_n)$ which satisfies (\ref{decrease}).
Moreover, the sequence $\left( u_n\right)_{n\in\mathbb{N}^*}$ strongly converges in $V$ towards $u$.
\end{thm}
\normalfont

\vspace{0.1cm}

\begin{proof}
 The proof is similar to the proof of Theorem~2.1 given in Section~3 except for Proposition~3.3 which gives estimate 
(\ref{estimate}):
$$\forall (r,s)\in V_t\times V_x, \; \; |\langle \mathcal{E}'(u_n) , r\otimes s\rangle| \leq A\|r_{n+1}\otimes s_{n+1}\|_V \|r\otimes s\|_V.$$

This estimate is no longer true, but we have a similar result which will be enough to complete the proof. Indeed, 
 let us prove that there exists a constant $B\in\mathbb{R}_+$ such that
\begin{equation}
\label{estimate2}
 \forall n\in\mathbb{N}^* , \; \forall (r,s)\in V_t\times V_x, \; \; |\langle \mathcal{E}'(u_n) , r\otimes s\rangle| \leq B \|r_n\otimes s_n\|_V \|r\otimes s\|_V.
\end{equation}

Let $M\in\mathbb{R}_+$ such that for all $n\in\mathbb{N}^*$, $\|u_n\|_V \leq M$. Its existence is ensured by Lemma~3.3. Let $K=\overline{B}(0,M+2)$ be the closed ball of $V$ centered at $0$ and of radius $M+2$. Let $L$ be the Lipschitz 
constant associated to $K$ in (\ref{Lipschitz}).

Let $(r,s)\in V_t\times V_x$ and $n \in \mathbb{N}^*$. As $(r_n,s_n)$ is a local minimum of 
$V_t \times V_x \ni (y,z) \mapsto \mathcal{E}\left(\sum_{k=1}^{n-1} r_k\otimes s_k + y\otimes z\right)$, there exists a constant $\eta \in (0,1)$ such that 
for all $\varepsilon \in (0,\eta)$, we have
\begin{equation}
\label{locmin}
\mathcal{E}\left(u_{n-1} + (r_n+\varepsilon r)\otimes (s_n+\varepsilon s)\right) \geq \mathcal{E}\left(u_{n-1} + r_n\otimes s_n\right).
\end{equation}

Moreover, by convexity of the functional $\mathcal{E}$, we have the following inequality
\begin{equation}
\label{convex}
\begin{array}{c}
\mathcal{E}\left(u_{n-1} + (r_n+\varepsilon r)\otimes (s_n+\varepsilon s)\right) - \mathcal{E}\left(u_{n-1} + r_n\otimes s_n\right) \\
\leq  \langle\mathcal{E}'(u_n + \varepsilon (r_n\otimes s+r\otimes s_n) + \varepsilon^2 r\otimes s) , \varepsilon (r_n\otimes s+r\otimes s_n) + \varepsilon^2 r\otimes s\rangle. \\
\end{array}
\end{equation}

We deduce from (\ref{locmin}), (\ref{convex}) and property (\ref{Lipschitz}) that, for all $\varepsilon$ small enough so that 
$\|\varepsilon (r_n\otimes s+r\otimes s_n) + \varepsilon^2 r\otimes s\|_V \leq 1$,
\begin{eqnarray*}
0  &\leq &\langle\mathcal{E}'(u_n + \varepsilon (r_n\otimes s+r\otimes s_n) + \varepsilon^2 r\otimes s) , \varepsilon (r_n\otimes s+r\otimes s_n) + \varepsilon^2 r\otimes s\rangle, \\
 &\leq &\langle\mathcal{E}'(u_n) , \varepsilon (r_n\otimes s+r\otimes s_n) + \varepsilon^2 r\otimes s\rangle + L \| \varepsilon (r_n\otimes s+r\otimes s_n) + \varepsilon^2 r\otimes s\|_V^2 . \\
\end{eqnarray*}

As $(r_n,s_n)$ is a local minimum of the functional $V_t\times V_x \ni (y,z) \mapsto \mathcal{E}\left(\sum_{k=1}^{n-1} r_k\otimes s_k + y\otimes z\right)$, $(r_n,s_n)$ still obeys the 
Euler equation (\ref{EL}) and thus $\langle\mathcal{E}'(u_n) , \varepsilon (r_n\otimes s+r\otimes s_n)\rangle = 0$.

Finally, we have
$$\varepsilon^2 \langle\mathcal{E}'(u_n) ,  r\otimes s\rangle + L \varepsilon^2 \| r_n\otimes s+r\otimes s_n + \varepsilon r\otimes s\|_V^2 \geq 0.$$

Dividing this expression by $\varepsilon^2$ and letting $\varepsilon$ go to zero, we obtain
$$\langle\mathcal{E}'(u_n) ,  r\otimes s\rangle + L \| r_n\otimes s+r\otimes s_n \|_V^2 \geq 0,$$

which leads to 
$$|\langle\mathcal{E}'(u_n) ,  r\otimes s\rangle |\leq  L \left( \| r_n\otimes s+r\otimes s_n \|_V^2 + \| r_n\otimes s-r\otimes s_n \|_V^2\right). $$

All this holds without the additional assumption (\ref{SVD}) for all $(r,s)\in V_t \times V_x$. To derive estimate (\ref{estimate2}), we
 use the additional assumption we made on $\|.\|_V$:
\begin{eqnarray*}
|\langle\mathcal{E}'(u_n) ,  r\otimes s\rangle |^{1/2} & \leq & \sqrt{L} \left( \|r_n\otimes s+r\otimes s_n\|_V^2 + \|r_n\otimes s-r\otimes s_n\|_V^2\right)^{1/2}, \\
& \leq &  \sqrt{L} \left( \|r_n\otimes s+r\otimes s_n\|_V + \|r_n\otimes s-r\otimes s_n\|_V\right), \\
& \leq & 2\sqrt{L}\left(\|r_n\otimes s\|_V + \|r\otimes s_n\|_V\right), \\
& \leq & 2\sqrt{L}\gamma \left(\|r_n\|_{V_t} \|s\|_{V_x} + \|r\|_{V_t}\|s_n\|_{V_x}\right).\\
\end{eqnarray*}

We can then choose $(r_n^*, s_n^*)\in V_t \times V_x$ and $(r^*, s^*) \in V_t \times V_x$ such that $r_n^* \otimes s_n^* = r_n \otimes s_n$ and $r^*\otimes s^* = r\otimes s$ 
 and such that $\|r_n^*\|_{V_t} = \|s_n^*\|_{V_x} \leq \sqrt{\frac{1}{\beta} \|r_n\otimes s_n\|_V}$ and
 $\|r^*\|_{V_t} = \|s^*\|_{V_x} \leq \sqrt{\frac{1}{\beta} \|r\otimes s\|_V}$. 
Thus,
\begin{eqnarray*}
|\langle\mathcal{E}'(u_n) ,  r\otimes s\rangle |^{1/2} & = & |\langle\mathcal{E}'(u_n) ,  r^*\otimes s^*\rangle |^{1/2},\\
& \leq & 2\sqrt{L}\gamma\left(\|r_n^*\|_{V_t} \|s^*\|_{V_x} + \|r^*\|_{V_t}\|s_n^*\|_{V_x}\right),\\
& \leq & 4\frac{\sqrt{L}\gamma}{\beta}\|r_n\otimes s_n\|_V^{1/2}\|r\otimes s\|_V^{1/2}.\\
\end{eqnarray*}

And in the end, we obtain estimate (\ref{estimate2}) with $B = 16 \frac{L\gamma ^2}{\beta^2}$. With this result, it is then possible to conclude 
as in the proof of Theorem~2.1.
\end{proof}

\begin{rem}{\normalfont Problem (\ref{penalized}) falls into the scope of Theorem~5.1. On the other hand, this is not the case for problem (\ref{poisson}), for which property (\ref{SVD}) is 
not true (see Remark~2.5). We were not able to prove a similar result in the case when $\|.\|_V$ does not satisfy property (\ref{SVD}).}
\end{rem}

\begin{rem}{
Here are two typical examples for which assumption $(A5)$ holds :
\begin{itemize} 
\item In the case when $V = V_t \otimes V_x$, property (\ref{SVD}) holds with $\beta=\gamma= 1$. 
This holds in uncertainty propagation problems where $V = L^2_T(\mathcal{T}, H)$ with $H$ an Hilbert space of real-valued functions defined on 
$\mathcal{X}$. Denoting by $V_t = L^2_T(\mathcal{T}, \mathbb{R})$ and $V_x = H$, then $V = V_t\otimes V_x$.
\item In other cases, to find an approximation of the global 
minimum of the energy $\mathcal{E}$, the Hilbert spaces $V_t$ and $V_x$ are usually discretized in finite-dimensional spaces. The problem can then be rewritten as a problem 
over $V_t = \mathbb{R}^l$, $V_x = \mathbb{R}^m$ with $l,m\in\mathbb{N}^*$, and then $V$ is naturally defined as the Hilbert space $V = \mathbb{R}^{l\times m}$. Then, assumptions $(A1)$, 
$(A2)$, $(A3)$ and $(A4)$ are automatically satisfied on the discretized spaces. As all the norms are equivalent in finite dimension,
the norms on $\mathbb{R}^l$, $\mathbb{R}^m$ and $\mathbb{R}^{l\times m}$ induced by the norms 
defined over the original Hilbert spaces $V_t$, $V_x$ and $V$  
are equivalent to the Frobenius norms, defined by (\ref{normkl}). These norms satisfy property (\ref{SVD}) since for all $(R,S)\in\mathbb{R}^l \times \mathbb{R}^m$, $\|RS^T\|_{lm} = \|R\|_l \|S\|_m$.
 Hence, the norms induced 
by the norms defined on the original Hilbert spaces automatically satisfy property (\ref{SVD}) even if the property is not satisfied in the continuous spaces.
\end{itemize}
}
\end{rem}

As in Section~4, we can prove that the algorithm defined by (\ref{localalgo}) and (\ref{decrease}) converges exponentially fast in finite dimension.

\begin{thm}
 Let us consider the algorithm defined by (\ref{localalgo}) and (\ref{decrease}). Let $l,m\in\mathbb{N}^*$. Let $V_t =\mathbb{R}^l$, $V_x = \mathbb{R}^m$ and $V = \mathbb{R}^{l\times m}$.
 Then there exist two constants $\tau>0$ and $\sigma \in (0,1)$ such that for all $n\in\mathbb{N}^*$,
\begin{equation}
\label{exponential}
0\leq \mathcal{E}(u_n) -\mathcal{E}(u) \leq \tau \sigma^n,
\end{equation}
and
\begin{equation}
 \label{vector}
\|u-u_n\|_V \leq \sqrt{\frac{2\tau}{ \alpha}} \sigma^{n/2}.
\end{equation}

\end{thm}

\begin{proof}
 As the spaces are finite-dimensional, assumptions $(A1)$, $(A2)$, $(A3)$, $(A4)$ and $(A5)$ are automatically fulfilled (see Remark~5.2) and estimate (\ref{estimate2}) holds true. The proof is similar to the proof of 
Theorem~4.1. Indeed, (\ref{ratebeg}), (\ref{init}), (\ref{un2}) and (\ref{un1}) still hold. Then it is sufficient to prove an inequality similar to (\ref{kappa}) to 
prove Theorem~5.2. However, as (\ref{estimate2}) holds instead of (\ref{estimate}), inequality (\ref{normeprime}) 
is replaced by:
\begin{equation}
\label{normeprime2}
 \|\mathcal{E}'(u_n)\|_V \leq \sqrt{lm}B\|r_{n} \otimes s_{n}\|_V^2,
\end{equation}
and consequently, an inequality similar to (\ref{kappa}) must be obtained in another way.

Let $M\in\mathbb{R}_+$ such that for all $n\in\mathbb{N}^*$, $\|u_n\|_V \leq M$.
 Its existence is ensured by Lemma~3.3. 
Let $K=\overline{B}(0, M+2+\|u\|_V)$ be the closed ball of~$V$ centered at $0$ of radius $M+2+\|u\|_V$. Let $L$ be the Lipschitz constant associated with~$K$ in (\ref{Lipschitz}). 

On the one hand, using the convexity of $\mathcal{E}$, (\ref{Lipschitz}) and the fact that $\mathcal{E}'(u) = 0$,  we have
\begin{equation}
\label{sup}
 \mathcal{E}(u_{n-1}) - \mathcal{E}(u_n) \leq -\langle\mathcal{E}'(u_{n-1}), r_n\otimes s_n\rangle \leq L\|r_n\otimes s_n \|_V \|u-u_{n-1}\|_V.
\end{equation}

On the other hand, (\ref{un1}), the convexity of $\mathcal{E}$, (\ref{EL}) and (\ref{normeprime2}) yield
\begin{eqnarray*}
\mathcal{E}(u_{n-1}) - \mathcal{E}(u_n)& = & \mathcal{E}(u_{n-1}) - \mathcal{E}(u) + \mathcal{E}(u) -\mathcal{E}(u_n) \\
& \geq & \frac{\alpha}{2}\|u-u_{n-1}\|^2_V + \mathcal{E}(u) -\mathcal{E}(u_n) \\
& \geq & \frac{\alpha}{2}\|u-u_{n-1}\|^2_V + \langle\mathcal{E}'(u_n), u-u_n\rangle\\
& = & \frac{\alpha}{2}\|u-u_{n-1}\|^2_V + \langle\mathcal{E}'(u_n), u-u_{n-1}\rangle\\
& \geq & \frac{\alpha}{2}\|u-u_{n-1}\|^2_V - \sqrt{lm}B\|r_n\otimes s_n\|_V \|u-u_{n-1}\|_V.\\
\end{eqnarray*}

Then, using (\ref{sup}), we have (\ref{kappa}) with $\kappa = \frac{\alpha}{2\left(L+\sqrt{lm}B\right)} \in (0,1)$.
We can conclude as in the proof of Theorem~4.1.
\end{proof}

\begin{rem}{\normalfont The results given in this section may not stand when we consider more than two Hilbert spaces. Indeed, the scheme of the proof of Theorem~5.1 cannot be easily adapted
 and we do not necessarily have an estimate similar to (\ref{estimate2}).}
\end{rem}

\section{Numerical results}

In this section, we describe how we implemented the algorithm introduced in Section~2 for the resolution of problem (\ref{penalized}) in a very simple setting, 
namely a one-dimensional membrane problem with uncertainty. We
 present the numerical results we obtained. Additional investigations to demonstrate the applicability and the efficiency of the procedure on high-dimensionnal problems are 
still required. We however refer to Nouy\cite{Nouy09} 
for illustrations of the interest of the method for problems in high dimensions.

\subsection{Implementation of the algorithm}
Let us recall problem (\ref{penalized}). Let $f\in L^2_T(\mathcal{T}, H^{-1}(\mathcal{X}))$ and $g\in L^2_T(\mathcal{T}, H^1_0(\mathcal{X}))$. Let us assume that the random variable 
$T$ has a probability density $p(t)$ for $t\in\mathcal{T}$. In other words, 
$$
\mathbb{P}(T \in \mathcal{D}) = \int_{\mathcal{D}} p(t)\,dt,
$$
where $\mathcal{D}$ is a measurable subset of $\mathcal{T}$. 

For a given value of the penalization parameter $\rho\in\mathbb{R}_+$, we wish to calculate an approximation of the minimizer of

 \begin{equation*}
  \mathop{\inf}_{v\in L^2_T(\mathcal{T}, H^1_0(\mathcal{X}))} \mathcal{E}(v),
 \end{equation*}
where $$\mathcal{E}(v) = \mathbb{E} \left[ \frac{1}{2}\int_{\mathcal{X}} |\nabla_x v(T,x)|^2 \,dx - \left\langle f(T,.), v(T,.) \right\rangle_{H^{-1}(\mathcal{X}), H^1_0(\mathcal{X})} + 
\frac{\rho}{2} \int_{\mathcal{X}} [g(T,x) -v(T,x)]_+^2 dx \right].$$

In other words, 
\begin{eqnarray*}
 \mathcal{E}(v)&  = & \frac{1}{2}\int_{\mathcal{X}\times \mathcal{T}} |\nabla_x v(t,x)|^2 p(t)\,dx\,dt - \int_{\mathcal{T}}\left\langle f(t,.), v(t,.) \right\rangle_{H^{-1}(\mathcal{X}), H^1_0(\mathcal{X})} p(t)\,dt\\
& &  + \frac{\rho}{2} \int_{\mathcal{X}\times \mathcal{T}} [g(t,x) -v(t,x)]_+^2 p(t) \,dx\,dt.\\
\end{eqnarray*}

In this case, our algorithm can be rewritten in the following form. Set $f_0 = f$ and $g_0=g$ and define recursively 
$(r_n,s_n) \in  L^2_T(\mathcal{T}, \mathbb{R}) \times H^1_0(\mathcal{X})$ as
$$(r_n,s_n)  \in  \mathop{\mbox{argmin}}_{(r,s)\in L^2_T(\mathcal{T}, \mathbb{R}) \times H^1_0(\mathcal{X})}  \mathcal{E}_n(r\otimes s),$$
with
\begin{eqnarray*}
 \mathcal{E}_n(r\otimes s) & = & \frac{1}{2}\int_{\mathcal{X}\times \mathcal{T}} |\nabla_x \left(r\otimes s\right)(t,x)|^2 p(t)\,dx\,dt- \int_{\mathcal{T}}\left\langle f_{n-1}(t,.), r\otimes s(t,.) \right\rangle_{H^{-1}(\mathcal{X}), H^1_0(\mathcal{X})} p(t)\,dt \\
& & +\frac{\rho}{2} \int_{\mathcal{X}\times \mathcal{T}} [g_{n-1}(t,x) - r\otimes s(t,x)]_+^2p(t)\,dx\,dt,\\
\end{eqnarray*}
where 
\begin{eqnarray*}
f_n& = &f_{n-1} + \Delta_x (r_n\otimes s_n),\\
g_n &= &g_{n-1} - r_n\otimes s_n.\\ 
\end{eqnarray*}

Indeed, 
$$\mathcal{E}\left(u_{n-1} + r\otimes s\right) = \mathcal{E}\left(u_{n-1}\right) - \frac{\rho}{2}\int_{\mathcal{X}\times \mathcal{T}} [g_{n-1}(t,x)]_+^2p(t)\,dx\,dt +\mathcal{E}_n(r\otimes s),$$
where $u_n$ is defined as in (\ref{undef}).

In fact, from Theorem 5.1, it is sufficient for $(r_n,s_n)$ to be a \em local \normalfont minimum of 
$ L^2_T(\mathcal{T}, \mathbb{R}) \times H^1_0(\mathcal{X}) \ni (r,s) \mapsto \mathcal{E}_n(r\otimes s)$ such that:
$$\mathcal{E}_n(r_n\otimes s_n) < \frac{\rho}{2} \int_{\mathcal{X}\times \mathcal{T}} [g_{n-1}(t,x)]_+^2 p(t)\,dx\,dt,$$
which ensures (\ref{decrease}).

\vspace{0.5cm}

We write the algorithm in the discrete case, and, for clarity, we restrict ourselves to the case of two open intervals $\mathcal{T}$ and $\mathcal{X}$ of $\mathbb{R}$. More precisely,
 $\mathcal{T} = ]\underline{t}, \overline{t}[$ and $\mathcal{X} = ]\underline{x}, \overline{x}[$, with $\underline{t}, \overline{t},\underline{x}, \overline{x} \in \mathbb{R}$, such that 
$\underline{t}<\overline{t}$ and $\underline{x}<\overline{x}$.
 
Let $l,m\in\mathbb{N}^*$, which will denote respectively the number of degrees of freedom in the discretized spaces of $V_t$ and $V_x$. Let us introduce a subdivision $(t_i)_{1\leq i \leq l}$ such that 
$\underline{t} = t_1 < t_2 <\cdots < t_{l-1} < t_{l} = \overline{t}$ and a subdivision $(x_j)_{j=1}^m$ such that $\underline{x} = x_0 < x_1 < \cdots < x_m < x_{m+1} = \overline{x}$. 

Let $(\phi_i)_{1\leq i \leq l}\subset V_t$ and $(\psi_j)_{1\leq j\leq m}\subset V_x$ 
be functions such that
\begin{equation}\label{eq:phii}
\phi_i(t_{i'}) = \delta_{ii'}, \; \forall 1\leq i,i' \leq l,
\end{equation}
and
\begin{equation}\label{eq:psij}
\psi_j(x_{j'}) = \delta_{jj'}, \; \forall 1\leq j \leq m, \; 0\leq j' \leq m+1,
\end{equation}
and let us consider $\widetilde{V}_t = \mbox{Span}(\phi_i)_{1\leq i\leq l}$ and $\widetilde{V}_x = \mbox{Span}(\psi_j)_{1\leq j\leq m}$. For example, $\mathbb{P}_q$ or $\mathbb{Q}_q$ finite 
elements satisfy these properties for all $q\in\mathbb{N}^*$.

Our goal is to find an approximation of the function $u$ under the following form
\begin{equation}\label{eq:discrete}
u(t,x) \approx \sum_{i=1}^l \sum_{j=1}^m U^{ij} \phi_i(t)\psi_j(x),
\end{equation}
where $U = (U^{ij})_{1\leq i\leq l, 1\leq j\leq m} \in \mathbb{R}^{k\times l}$. 

Let $D\in\mathbb{R}^{m \times m}$ be the symmetric positive definite matrix which corresponds to the discretization of the one-dimensional operator
 $-\partial_{xx}$ in $V_x$, in other words, for all $1\leq j,j' \leq m$, 
$$D^{jj'} = \int_{\mathcal{X}} \partial_x\psi_j(x)\partial_x\psi_{j'}(x)\,dx.$$

Let also $\Phi = (\Phi^{ii'})_{1\leq i,i' \leq l}\in\mathbb{R}^{l\times l}$ and $\Psi = (\Psi^{jj'})_{1\leq j,j' \leq m}\in\mathbb{R}^{m\times m}$ be the symmetric positive definite matrices defined as
$$
\Phi^{ii'} = \int_{\mathcal{T}} \phi_i(t)\phi_{i'}(t)p(t)\,dt, \; \forall 1\leq i,i'\leq l,
$$
and
$$
\Psi^{jj'} = \int_{\mathcal{X}} \psi_j(x)\psi_{j'}(x)\,dx, \; \forall 1\leq j,j'\leq m.
$$

With discretization (\ref{eq:discrete}), the term $\frac{1}{2}\int_{\mathcal{X}\times \mathcal{T}} |\nabla_x u(t,x)|^2 p(t)\,dx\,dt$ is then equal to
$$
\frac{1}{2}\int_{\mathcal{X}\times \mathcal{T}} |\nabla_x u(t,x)|^2 p(t)\,dx\,dt \approx\mbox{Tr}(\Phi UDU^T).
$$

Similarly, if we denote by $F = (F^{ij})_{1\leq i\leq l, 1\leq j\leq m} \in \mathbb{R}^{l\times m}$ the matrix defined as, for all $1\leq i\leq l$ and $1\leq j\leq m$, 
$$
F^{ij} = \int_{\mathcal{T}} \left\langle f(t,.), \phi_i\otimes\psi_j(t,.) \right\rangle_{H^{-1}(\mathcal{X}), H^1_0(\mathcal{X})} p(t)\,dt,
$$
the term $\int_{\mathcal{T}} \left\langle f(t,.), u(t,.) \right\rangle_{H^{-1}(\mathcal{X}), H^1_0(\mathcal{X})} p(t)\,dt$ is approximated by
$$
\int_{\mathcal{T}} \left\langle f(t,.), u(t,.) \right\rangle_{H^{-1}(\mathcal{X}), H^1_0(\mathcal{X})} p(t)\,dt \approx \mbox{Tr}(FU^T).
$$

The approximation of the term $\frac{\rho}{2} \int_{\mathcal{X}\times \mathcal{T}} [g(t,x) -u(t,x)]_+^2 p(t) \,dx\,dt$ is more subtle. Indeed, let us approximate the function $g(x,t)$ in 
the discretized space $\widetilde{V}_t \otimes \widetilde{V}_x$ by
$$
g(t,x) \approx \sum_{i=1}^l \sum_{j=1}^m G^{ij} \phi_i(t)\psi_j(x).
$$
Given relationships (\ref{eq:phii}) and (\ref{eq:psij}), a natural way to define $G^{ij}$ is the following
$$
G^{ij} = g(t_i, x_j) \; \forall 1\leq i \leq l, \; 1\leq j\leq m.
$$

It then holds
$$
 \frac{\rho}{2} \int_{\mathcal{X}\times \mathcal{T}} [g(t,x) -u(t,x)]_+^2 p(t) \,dx\,dt  \approx  
\frac{\rho}{2} \int_{\mathcal{X}\times \mathcal{T}} \left[ \sum_{i=1}^l \sum_{j=1}^m (G^{ij} - U^{ij}) \phi_i(t)\psi_j(x) \right]_+^2p(t)\,dx\,dt.
$$

Let $w$ be the function defined as $w:(t,x)\in\mathcal{T}\times \mathcal{X} \mapsto w(t,x) = \left[\sum_{i=1}^l \sum_{j=1}^m (G^{ij} - U^{ij}) \phi_i(t)\psi_j(x)\right]_+$. From assumptions (\ref{eq:phii}) and (\ref{eq:psij}), it holds
$$
w(t_i, x_j) = [G^{ij} - U^{ij}]_+, \; \forall 1\leq i \leq l,\; 1\leq j \leq m.
$$
Thus, we perform a supplementary approximation, that is to approximate the function $w$ itself as
\begin{eqnarray*}
w(t,x) & \approx & \sum_{i=1}^l \sum_{j=1}^m w(t_i,x_j) \phi_i(t)\psi_j(x),\\
& = & \sum_{i=1}^l \sum_{j=1}^m [G^{ij} - U^{ij}]_+ \phi_i(t)\psi_j(x).\\
\end{eqnarray*}

Finally, it holds
\begin{eqnarray*}
 \frac{\rho}{2} \int_{\mathcal{X}\times \mathcal{T}} [g(t,x) -u(t,x)]_+^2 p(t) \,dx\,dt & \approx &\frac{\rho}{2} \int_{\mathcal{X}\times \mathcal{T}} \left(\sum_{i=1}^l \sum_{j=1}^m \left[G^{ij} - U^{ij} \right]_+\phi_i(t)\psi_j(x)\right)^2 p(t)\,dx\,dt,\\
& = &  \mbox{Tr}(\Phi [G-U]_+ \Psi [G-U]_+^T).\\
\end{eqnarray*}
We also made the following mass lumping approximation, that is to consider that $\Phi \approx \mbox{Id}_{l}$ and $\Psi \approx \mbox{Id}_m$. 

\vspace{1 cm}

Then, the discretized problem (\ref{penalized}) can be rewritten as
\begin{equation*}
\begin{array}{l}
 \mbox{Find } U\in\mathbb{R}^{l\times m} \mbox{ such that }\\
 U \in \mathop{\mbox{argmin}}_{V\in \mathbb{R}^{l \times m}} \frac{1}{2} VD:V - F:V + \frac{\rho}{2} [G-V]_+:[G-V]_+,\\
\end{array}
\end{equation*}
where for $A,B\in\mathbb{R}^{l\times m}$, 
$$
A:B = \mbox{Tr}( A B^T) = \sum_{1\leq i \leq l} \sum_{1 \leq j \leq m} A_{ij} B_{ij}.
$$ 

This problem is equivalent to:
$$\mbox{Find } U\in\mathbb{R}^{l\times m} \mbox{ such that } UD = F + \rho[G-U]_+.$$

\vspace{0.1cm}
For each function $r\in V_t$ and $s\in V_x$, we denote by $R\in\R^l$ and $S\in \R^m$, their coordinates in the bases $(\phi_i)_{1\leq i\leq l}$ and $(\psi_j)_{1\leq j\leq m}$, which are 
given by
$$
\forall 1\leq i\leq l, \; R^i = \int_{\T} r(t)\phi_i(t)\,dt,
$$
and
$$
\forall 1\leq j\leq m, \; S^j = \int_{\X} s(x)\psi_j(x)\,dx.
$$
Our algorithm can then be rewritten as:

Choose a threshold $\varepsilon >0$ and set $F_0 = F$, $G_0 = G$. At iteration $n\geq 1$:
\begin{enumerate}
 \item find $R_n = (R_n^i)_{1\leq i\leq l}$ and $S_n = (S_n^j)_{1\leq j\leq m}$ two vectors respectively in $\mathbb{R}^l$ and $\mathbb{R}^m$ such that:
$$
(R_n, S_n) \in \mathop{\mbox{argmin}}_{(R,S)\in\mathbb{R}^l \times \mathbb{R}^m}\mathcal{E}_n(R,S),
$$
with 
\begin{eqnarray*}
 \mathcal{E}_n(R,S) & = & \frac{1}{2} (RS^T)D:(RS^T) - F_{n-1}:(RS^T) \\
& &  +\frac{\rho}{2} [G_{n-1} - RS^T]_+ :[G_{n-1} - RS^T]_+. \\
\end{eqnarray*}
\item set $F_n = F_{n-1} - (R_n S_n^T)D,$ and $G_n = G_{n-1} - R_nS_n^T$. 
\item if $\| F_n + \rho [G_n]_+\| \geq \varepsilon$, proceed to iteration $n+1$. Otherwise, stop.
\end{enumerate}

\vspace{0.1cm}

The remaining question is: how can we compute $(R_n,S_n)$ at step 1? This critical step is described in the following section.

\subsection{Computing $(R_n,S_n)$}

\subsubsection{Fixed-point procedure}

Let us first describe a method which has been proposed by Nouy\cite{Nouy09} and Chinesta\cite{Chinesta}, that is the fixed-point procedure and which we use in our final 
numerical implementation (see Section~6.2.2). We present this algorithm 
in a particular case. Let us consider $V_t = \mathbb{R}^l$, $V_x = \mathbb{R}^m$ and $V = \mathbb{R}^{l\times m}$ endowed with the Frobenius norms defined by 
(\ref{normkl}). 
We fix a given matrix $M\in\mathbb{R}^{l\times m}$. Let us define
 the energy functional as $\mathcal{E}(W) =  \|M-W\|_{V}^2$ for $W\in\mathbb{R}^{l\times m}$. In this particuler case, applying 
the greedy algorithm described above consists in computing the Singular Value Decomposition of the matrix $M$.

In this particular case, the greedy algorithm can be rewritten in the following form.

Choose a threshold $\varepsilon>0$ and set $M_0 = M$. At iteration $n\geq 1$,
\begin{enumerate}
 \item find two vectors $R_n$ and $S_n$ respectively in $\mathbb{R}^l$ and $\mathbb{R}^m$ such that
\begin{equation}
(R_n, S_n) \in \mathop{\mbox{argmin}}_{(R,S)\in\mathbb{R}^l \times \mathbb{R}^m} \left\|M_{n-1} - RS^T\right\|_{V}^2.
\end{equation}
\item set $M_n = M_{n-1} - R_n S_n^T$.
\item if $\| M_n \|_{V} \geq \varepsilon$, proceed to iteration $n+1$. Otherwise, stop.
\end{enumerate}

The Euler equation associated to this problem can be rewritten as
$$\left\{
\begin{array}{ccc}
 \|S_n\|_{V_x}^2 R_n & = &M_{n-1} S_n,\\ [5pt]
\|R_n\|_{V_t}^2 S_n & = & (M_{n-1})^T R_n.\\
\end{array} \right .
$$

The method which is generally used\cite{LeBr09} to solve these Euler equation is a fixed-point algorithm, which simply reads (for a fixed $n$):
 at iteration $q\geq 0$, compute two vectors $(R_n^{(q+1)}, S_n^{(q+1)})\in \mathbb{R}^k \times \mathbb{R}^l$ such that
\begin{equation}
\label{fixedpoint}
 \left\{
\begin{array}{ccc}
 \|S_n^{(q)}\|_{V_x}^2 R_n^{(q+1)} & = &M_{n-1} S_n^{(q)},\\[5pt]
\|R_n^{(q+1)}\|_{V_t}^2 S_n^{(q+1)} & = & (M_{n-1})^T R_n^{(q+1)}.\\
\end{array} \right . 
\end{equation}

One can check\cite{LeBr09} that this procedure is similar to the power method to compute the largest eigenvalue (and associated eigenvector) of the matrix $(M_{n-1})^TM_{n-1}$.

\vspace{0.5cm}
One could think of transposing this fixed-point procedure to the case of the obstacle problem we consider in this article. In our case, the Euler equation
$$\left\{
\begin{array}{ccc}
 (R_n:R_n) D S_n & = & F_{n-1}^T R_n + \rho [G_{n-1} - R_n S_n^T]_+^T R_n, \\[5pt]
(DS_n:S_n) R_n & = & F_{n-1} S_n + \rho [G_{n-1} - R_n S_n^T]_+ S_n .\\
\end{array} \right .
$$
could be solved {\it a priori} with a fixed point algorithm, which, at iteration $q$, might be written as
$$\left\{
\begin{array}{ccc}
 (R_n^{(q)}:R_n^{(q)}) D S_n^{(q+1)} & = & F_{n-1}^T R_n^{(q)} + \rho \left[G_{n-1} - R_n^{(q)} \left(S_n^{(q)}\right)^T\right]_+^T R_n^{(q)}, \\[5pt]
(DS_n^{(q+1)}:S_n^{(q+1)}) R_n^{(q+1)} & = & F_{n-1} S_n^{(q+1)} + \rho\left[G_{n-1} - R_n^{(q)} \left(S_n^{(q+1)}\right)^T\right]_+ S_n^{(q+1)}. \\
\end{array} \right .
$$

Unfortunately, we have not been able to make this
fully-explicit fixed point algorithm converge for large values of the parameter $\rho$. We therefore have
decided to resort to a minimization procedure.

\subsubsection{Minimization procedure}

\vspace{0.1cm}

The approach we adopt then is the following. We choose an initial pair $(R_n^0, S_n^0)\in\mathbb{R}^l\times\mathbb{R}^m$ and then perform a 
quasi-newton algorithm to find a local minimum of the function 
$$\frac{1}{2} \left(RS^T\right)D:\left(RS^T\right) - F_{n-1}:\left(RS^T\right)  +\frac{\rho}{2} \left[G_{n-1} - RS^T\right]_+ : \left[G_{n-1} - RS^T\right]_+.$$

The main difficulty is to find a proper initial pair $(R_n^{(0)}, S_n^{(0)})$ such that 
$$\frac{1}{2} \left(R_n^{(0)}S_n^{(0)T}\right)D:\left(R_n^{(0)}S_n^{(0)T}\right) - F_{n-1}:\left(R_n^{(0)}S_n^{(0)T}\right)  +\frac{\rho}{2} \left[G_{n-1} - R_n^{(0)}S_n^{(0)T}\right]_+ : \left[G_{n-1} - R_n^{(0)}S_n^{(0)T}\right]_+$$
$$ < \frac{\rho}{2} [G_{n-1}]_+ : [G_{n-1}]_+, $$
to ensure that the energy decreases (see (\ref{decrease})).
 
Let us describe our approach in the continuous setting with the notation used in Section~4. It consists in finding a pair $(r_n^{(0)}, s_n^{(0)})\in V_t\times V_x$ such that
$$\mathcal{E}\left(u_{n-1} + r_n^{(0)}\otimes s_n^{(0)} \right) < \mathcal{E}\left(u_{n-1} \right).$$

We notice that for $(r,s)\in V_t\times V_x$, and $\eta >0$, we have
$$\mathcal{E}\left(u_{n-1} + \eta r\otimes s \right) - \mathcal{E}\left(u_{n-1}\right) = \eta\left\langle \mathcal{E}'\left(u_{n-1}\right), r\otimes s\right\rangle + o(\eta),$$
for $\eta$ small enough.

The idea is then to find a pair $(r,s)\in V_t \times V_x$ such that $\langle \mathcal{E}'(u_{n-1}), r\otimes s\rangle <0$,
 so that there exists $\eta>0$ small enough 
for which $\mathcal{E}\left(u_{n-1} + \eta r\otimes s \right) - \mathcal{E}\left(u_{n-1}\right) <0$. Then, 
$r_n^{(0)}\otimes s_n^{(0)} = \eta r\otimes s$ is a good initial guess.

Let us first consider the pair $\left(\overline{r_n^{(0)}},\overline{s_n^{(0)}}\right)\in V_t\times V_x$ such that
$$\left(\overline{r_n^{(0)}}, \overline{s_n^{(0)}}\right) \in \mathop{\mbox{argmin}}_{(r,s)\in V_t\times V_x} \frac{1}{2}\left\|\mathcal{E}'\left(u_{n-1} \right) - r\otimes s\right\|_V^2.$$

In other words, we consider $\left(\overline{r_n^{(0)}},\overline{s_n^{(0)}}\right)$ the first term of the singular value decomposition of $\mathcal{E}'\left(u_{n-1}\right)$ in $V$. The 
Euler equations then imply
$$-\left\langle \mathcal{E}'(u_{n-1}) - \overline{r_n^{(0)}}\otimes \overline{s_n^{(0)}}, \overline{r_n^{(0)}}\otimes\overline{s_n^{(0)}}\right\rangle = 0,$$

and therefore,
$$\left \langle \mathcal{E}'(u_{n-1}), \overline{r_n^{(0)}}\otimes \overline{s_n^{(0)}}\right\rangle =  \left\| \overline{r_n^{(0)}}\otimes \overline{s_n^{(0)}}\right\|^2_V >0.$$

By taking $r_n^{(0)}\otimes s_n^{(0)} = -\eta\overline{r_n^{(0)}}\otimes \overline{s_n^{(0)}}$, there exists then $\eta >0$ small enough such that
$$\mathcal{E}\left(u_{n-1}+ r_n^{(0)}\otimes s_n^{(0)} \right) - \mathcal{E}\left(u_{n-1}\right) <0.$$

In the discrete case associated to problem (\ref{penalized}), $\left(\overline{R_n^{(0)}}, \overline{S_n^{(0)}}\right)$ is obtained by taking the first term of the singular value decomposition of
 the matrix $F_{n-1} +  \rho [G_{n-1} ]_+$. This can 
be done with a fixed point procedure similar to (\ref{fixedpoint}).

Once we have this initial guess $\left(R_n^{(0)}, S_n^{(0)}\right)$, we perform a quasi-newton algorithm to minimize the energy. The computations are done with the software Scilab\cite{Scilab} and the quasi-Newton 
procedure is performed via the \itshape optim \normalfont procedure of Scilab.

Let us point out that this procedure is intrusive in general.

\subsection{One-dimensional membrane problem}

In this section, we present the results we obtained with this algorithm on the following membrane problem. 

We suppose $\mathcal{X} = \mathcal{T} = (0,1)$. We consider a random variable $T$ following a uniform law 
of probability on the interval $(0,1)$. We wish to study problem (\ref{obstacle}) with the following values for $f$ and $g$,
$$\forall (t,x)\in(0,1)^2 \;, \; f(t,x) = -1 \mbox{ and } g(t,x) = t[\sin(3\pi x)]_+ + (t-1)[\sin(3\pi x)]_-.$$

The negative part of $a\in\mathbb{R}$, i.e. $[a]_- = 0$ if $a\geq 0$, and $[a]_- = -a$ if $a\leq 0$, is denoted by $[a]_-$.

The above problem models a rope attached at $x=0$ and $x=1$ subjected to gravity and resting upon obstacles whose altitudes are given by $g(t,x)$. The quantity $u(t,x)$ then represents the altitude of the rope at 
abscissa $x$ when $T=t$.

This problem is approximated by problem (\ref{penalized}) with parameter $\rho=2500$. The problem is discretized with a regular mesh and $\mathbb{P}_1$ finite elements in each 
direction. Discretization parameters are chosen as $l=m=40$.

Figure 2 represents the altitude of the obstacles given by $g(t,x)$ for $(t,x)\in[0,1]^2$.

\begin{figure}[h]
\begin{center}
\psfrag{g(t,x)}{$g(t,x)$}
\psfrag{t}{$t$}
\psfrag{x}{$x$}
 \includegraphics[width=8cm,bb=   73   252   521   589]{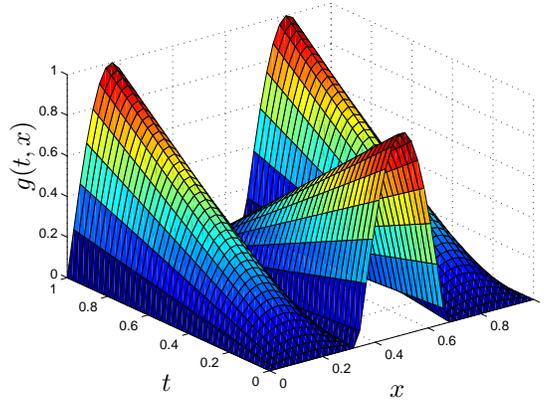}
  \caption{Altitude of the obstacles as a function of $t$ and $x$.}
\end{center}
\vspace*{8pt}
\end{figure}

The algorithm described in the previous sections is then applied with the following stopping criterion: $\|F_n+\rho[G_n]_+\|_V < 10^{-4}$ with 
$\|A\|_V = \sqrt{\mbox{Tr}(AA^T)} = \sqrt{\sum_{i=1}^k\sum_{j=1}^l A_{ij}^2}$ for $A\in\mathbb{R}^{k\times l}$.

Figure 3 represents the evolution of $\log_{10}\left(\mathcal{E}(u_n) - \mathcal{E}(u)\right)$ and of $\log_{10}(\|F_n+\rho[G_n]_+\|_V)$.

\begin{figure*}[h]
\begin{center}
\label{rate}
\psfrag{Number of iterations}[][][0.6]{Number of iterations}
\psfrag{log energy}[][][0.7]{$\log_{10}(\mathcal{E}(u_n) - \mathcal{E}(u))$}
 \includegraphics[width=5.9cm,bb=   73   252   521   589]{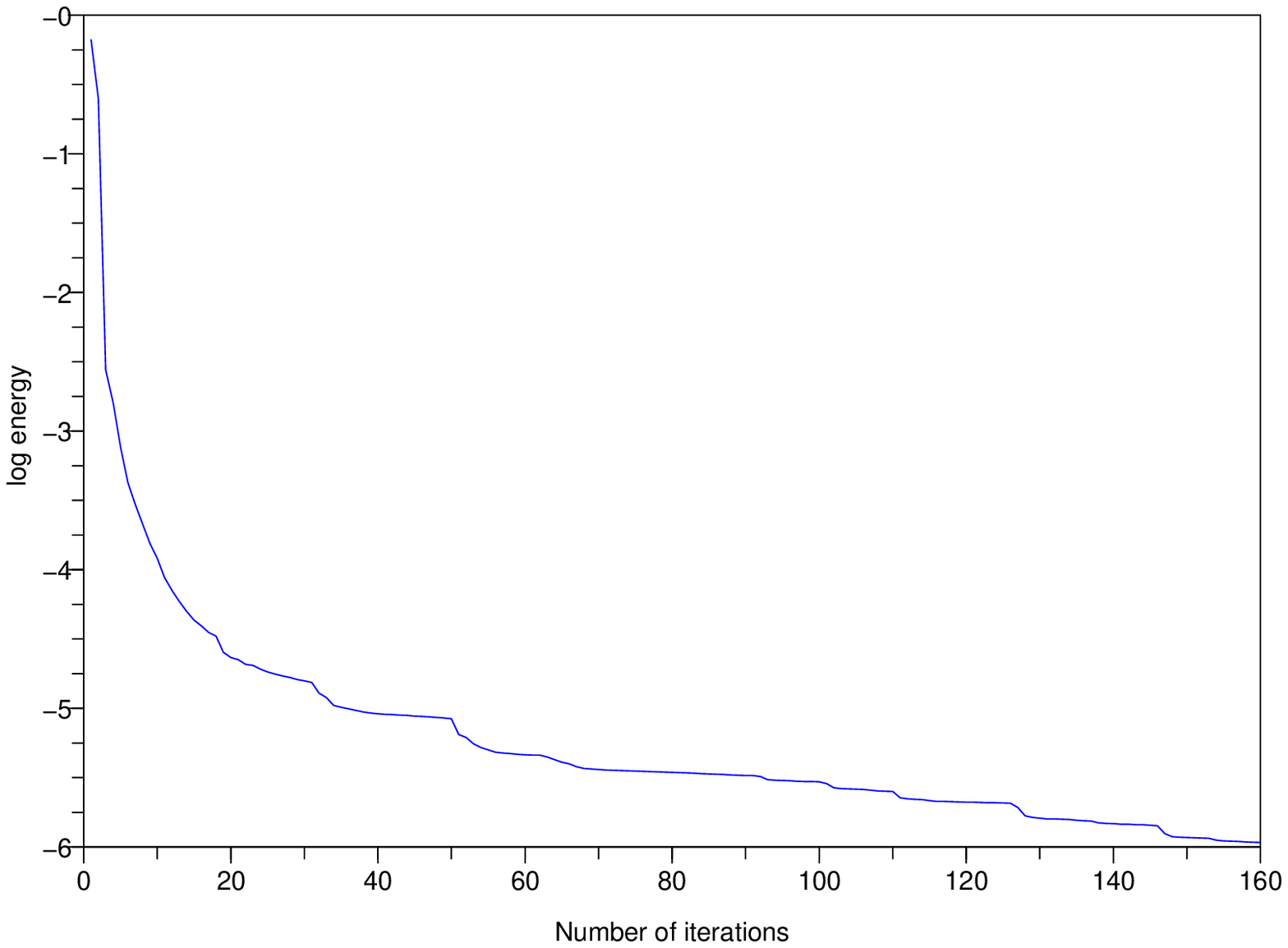}
\psfrag{Number of iterations}[][][0.6]{Number of iterations}
\psfrag{log limit}[][][0.7]{$\log_{10}(\|F_n + \rho [G_n]_+\|_V)$}
 \includegraphics[width=5.9cm,bb=   73   252   521   589]{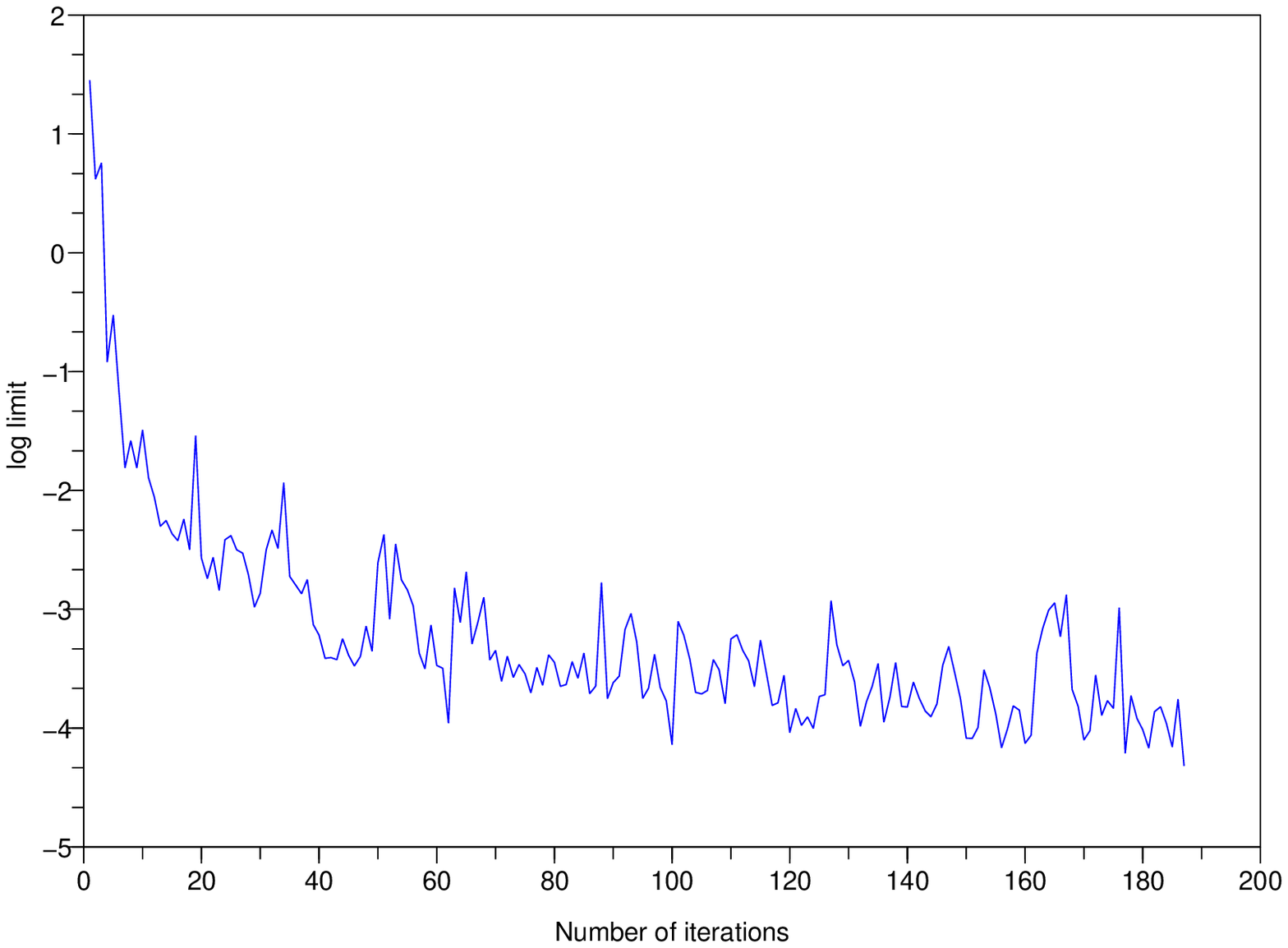}
  \caption{Evolution of $\log_{10}\left(\mathcal{E}(u_n) - \mathcal{E}(u)\right)$ (left) and of $\log_{10}(\|F_n+\rho[G_n]_+\|_V)$ (right) as a function of the number of terms $(R_n,S_n)$ computed.}
\end{center}
\vspace*{8pt}
\end{figure*}

We can see that our algorithm captures very quickly the main modes of the solution and that both the energy and the $V$-norm of the residue $\|F_n+\rho[G_n]_+\|_V$ converges
 exponentially fast, as predicted by Theorem 5.1.

Figure 4 represents the results obtained for the solution $u(t,x)$. Figure 5 and 6 represents $u(t,x)$ and $g(t,x)$ for some special values of $T$.

\begin{figure}[h]
\label{fig3}
\begin{center}
\psfragscanon
\psfrag{u(t,x)}{$u(t,x)$}
\psfrag{t}{$t$}
\psfrag{x}{$x$}
 \includegraphics[width=8cm,bb=   73   252   521   589]{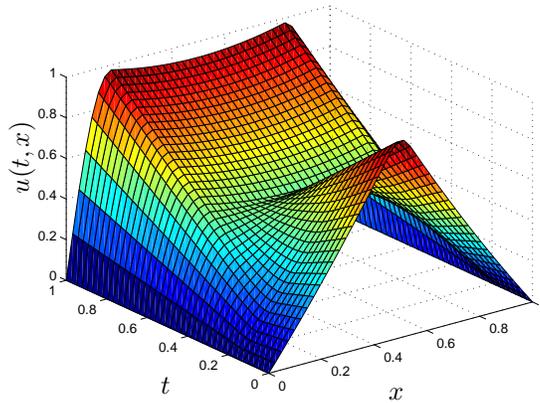}

\vspace*{8pt}
  \caption{Altitude of the rope as a function of $t$ and $x$.}
\end{center}
\vspace*{8pt}
\end{figure}

\begin{figure*}[h]
\begin{center}
\psfrag{u(x)}[][][0.4]{$u(x)$}
\psfrag{g(x)}[][][0.4]{$g(x)$}
 \includegraphics[width=5.9cm,bb=   73   252   521   589]{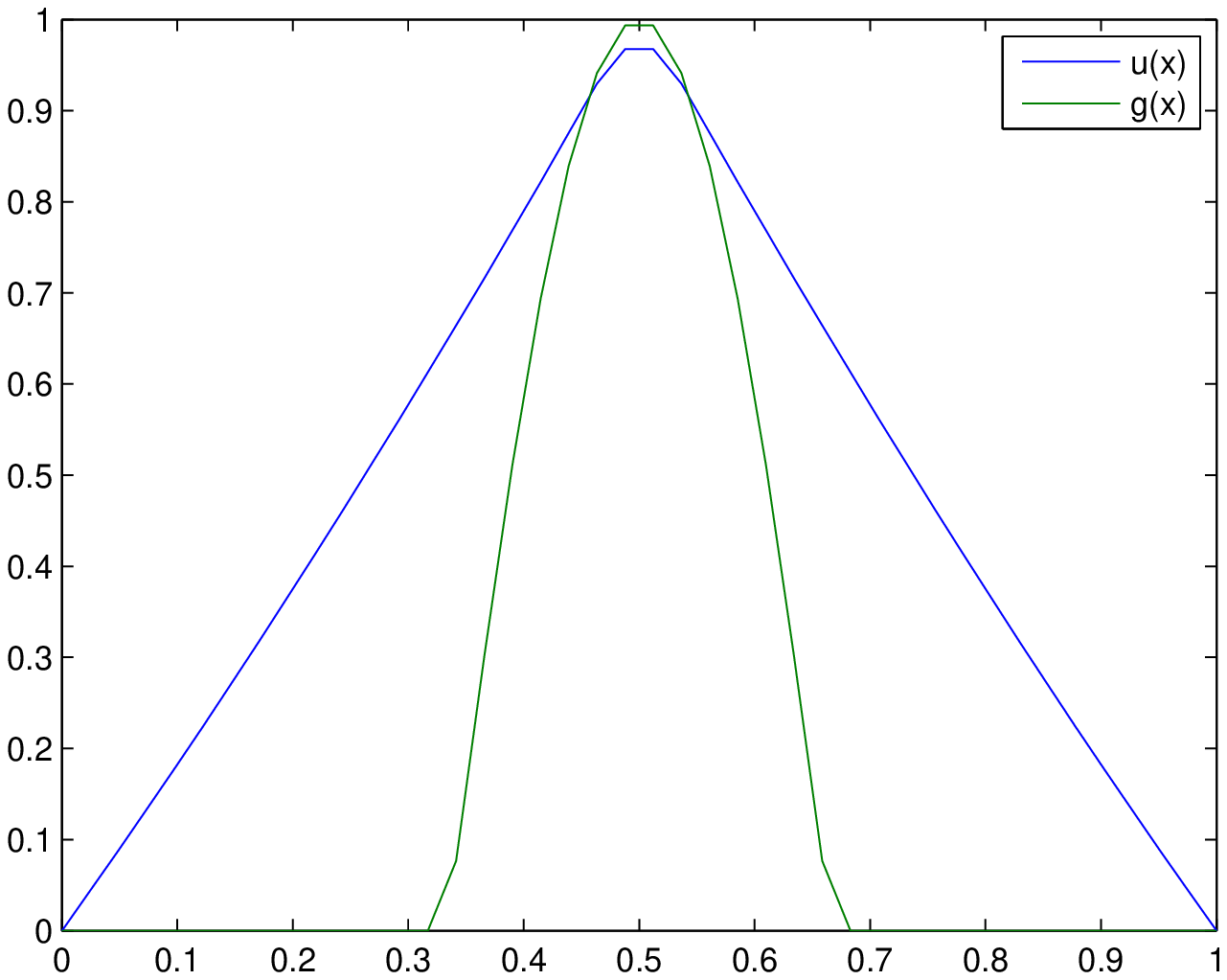}
 \includegraphics[width=5.9cm,bb=   73   252   521   589]{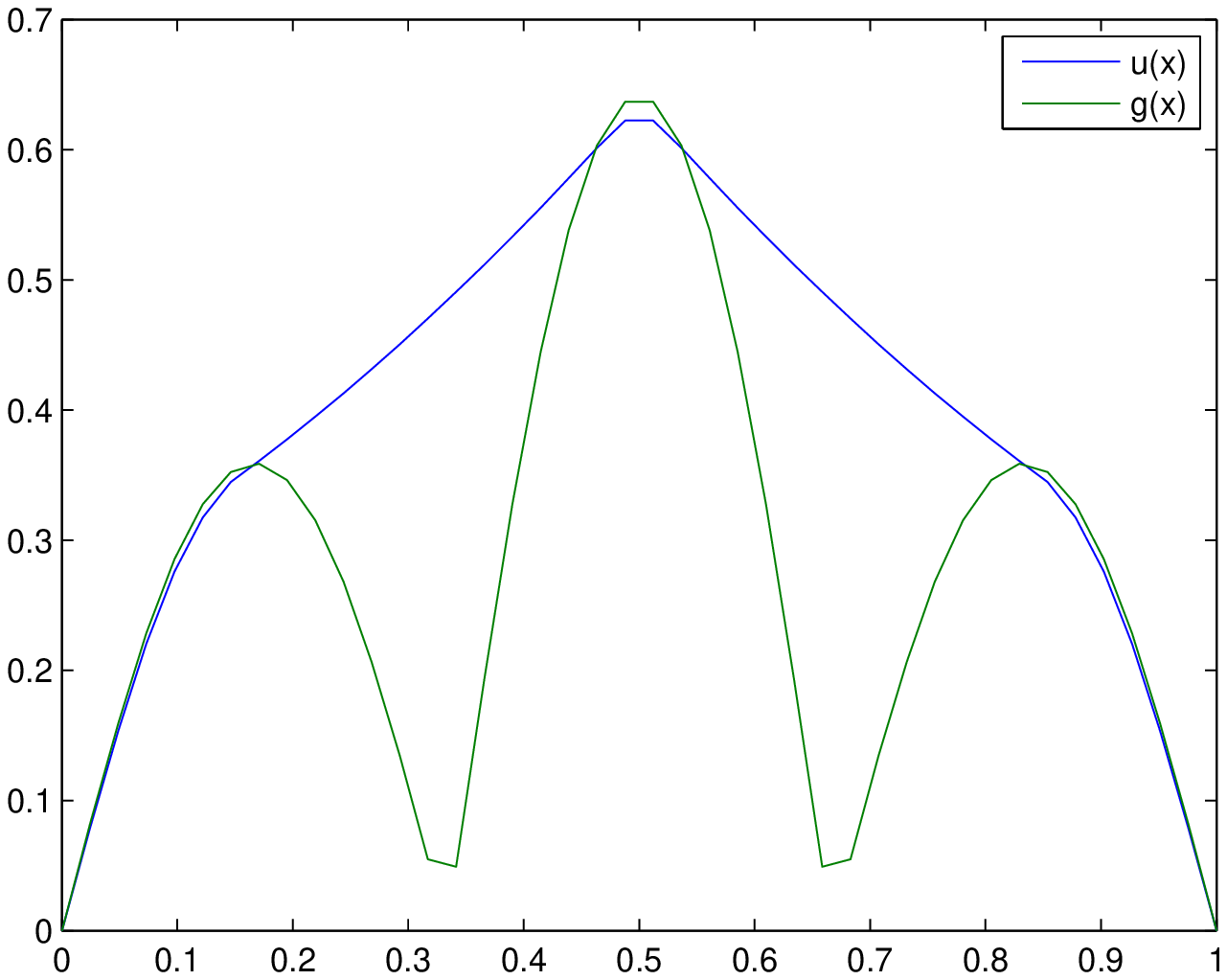}
  \caption{Profile of $u$ and $g$ for $T=0$ (left) and $T=0.375$ (right).}
\end{center}
\vspace*{8pt}
\end{figure*}

\begin{figure*}[h]
\begin{center}
\psfrag{u(x)}[][][0.4]{$u(x)$}
\psfrag{g(x)}[][][0.4]{$g(x)$}
 \includegraphics[width=5.9cm,bb=   73   252   521   589]{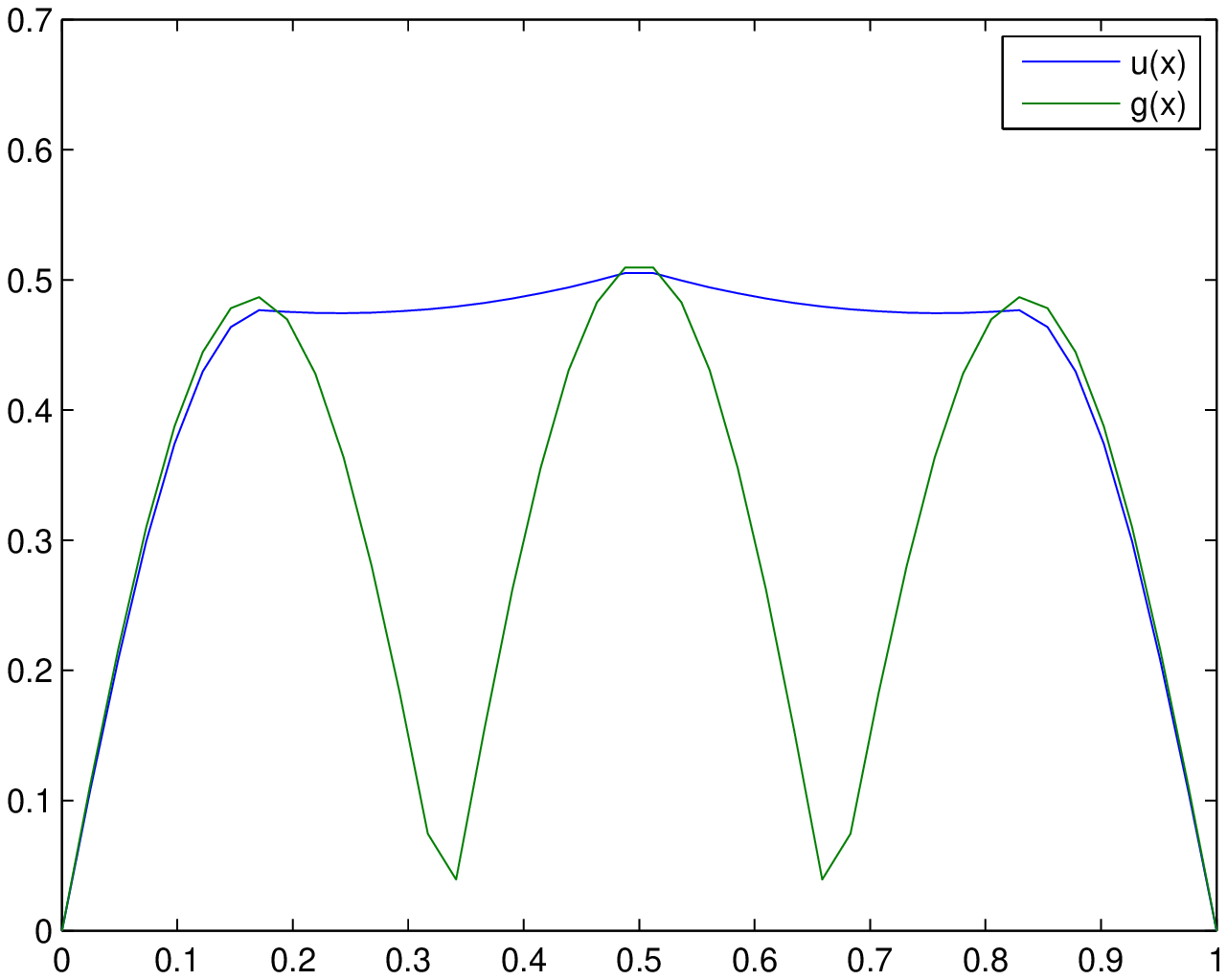}
 \includegraphics[width=5.9cm,bb=   73   252   521   589]{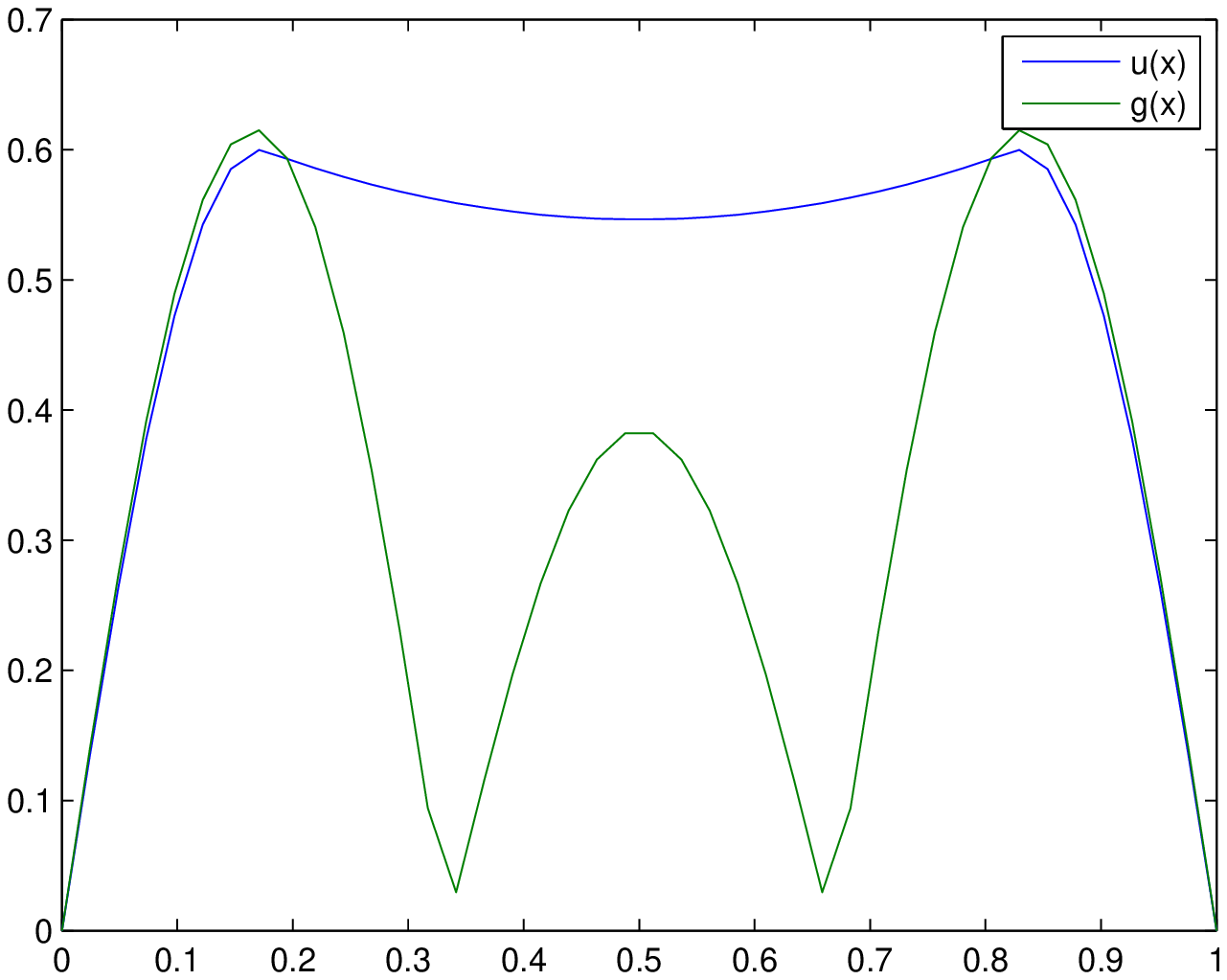}
  \caption{Profile of $u$ and $g$ for $T=0.5$ (left) and $T=0.625$ (right).}
\end{center}
\vspace*{8pt}
\end{figure*}

As we can observe, the solution does not exactly satisfies the constraint $u(t,x) \geq g(t,x)$. This is due to the fact that we approximate a solution 
$u_{\rho}$ of the penalized problem (\ref{penalized}) for $\rho=2500$. This is the main drawback of our method: we do not approximate directly the solution 
of the initial obstacle problem but the solution of a close regularized problem. Indeed, if we try to further increase the parameter $\rho$, we face the 
main drawback of penalization methods, that is the ill-conditioning of the resulting matrices.

\section{Conclusion}

In this article, we presented a greedy algorithm based on variable decomposition aiming at computing the global minimum of a strongly convex energy functional. We proved 
that, provided that the gradient of the energy is Lipschitz on bounded sets, and that the Hilbert spaces considered satisfy assumptions $(A1)$ and $(A2)$, then the approximation 
given by our algorithm strongly converges towards the desired result. One of the main advantage of the algorithm is that it can deal with highly nonlinear problems. We also proved that 
in finite dimension, this algorithm converges exponentially fast.

We applied this algorithm in the context of uncertainty quantification on obstacle problems. In this frame, we considered regularizations of this kind 
of problems by penalization methods. Indeed, the obstacle problem can be approximated by a global minimization problem defined on the entire Hilbert space of 
some strongly convex energy functional where the constraints of the initial problem are replaced by penalization terms in the expression of the functional. Our algorithm gives a good approximation 
of the solutions of the regularized problem. However, the problem of ill-conditioned matrices, which is inherent to penalization methods, limits the accuracy 
with which we can approach the solution of the initial obstacle problem.

One way to circumvent this problem is to use augmented Lagrangian methods (see Ref.~\cite{For82,Glow89,Gros07})  instead of penalization methods.
 Indeed, the former algorithms converge towards the true solution of the initial obstacle problems. The adaptation of
 our algorithm to such methods is work in progress.

Another extension of our work would be to consider other problems than obstacle problems. In Ref.~\cite{Nouy09}, a similar algorithm based on Proper 
Generalized Decomposition is used to study uncertainty quantification upon a Burger type equation. We believe that it could be possible to extend our 
proof of convergence in the case of such hyperbolic systems.

\section*{Acknowledgment}
We would like to thank the Michelin company for financial support.

\end{document}